\documentclass[11pt]{article}

\usepackage[a4paper,left=3.5cm,right=3.5cm,top=3cm,bottom=3cm]{geometry}
\usepackage{graphicx}%
\usepackage{multirow}%
\usepackage{amsmath,amssymb,amsfonts}%
\usepackage{mathrsfs}%
\usepackage[title]{appendix}%
\usepackage{xcolor}%
\usepackage{textcomp}%
\usepackage{manyfoot}%
\usepackage{booktabs}%
\usepackage{algorithm}%
\usepackage{algorithmicx}%
\usepackage{algpseudocode}%
\usepackage{listings}%
\usepackage{xcolor}
\usepackage{cancel}

\usepackage{mathtools}
\usepackage{setspace}
\usepackage{color}
\usepackage{soul}
\usepackage{amsthm}

\def\bq{\begin{equation}}
\def\eq{\end{equation}}
\def\bqy{\begin{eqnarray}}
\def\eqy{\end{eqnarray}}
\def\bal#1\eal{\begin{align}#1\end{align}}

\def\bfH{\mathbf{H}}

\def\bfV{\mathbf{V}}

\def\bfv{\mathbf{v}}
\def\bfu{\mathbf{u}}
\def\bfp{\mathbf{p}}

\def\bfp{\mathbf{p}}
\def\bfn{\mathbf{n}}

\def\bfxi{\boldsymbol{\xi}}

\def\md{\mathrm{d}}

\def\bq{\begin{equation}}
\def\eq{\end{equation}}
\def\bqy{\begin{eqnarray}}
\def\eqy{\end{eqnarray}}
\def\bal#1\eal{\begin{align}#1\end{align}}

\def\4tensLa{\bar{\bar{\Lambda}}}

\def\al{\alpha}

\def\De{\Delta}

\def\Ga{\Gamma}

\def\na{\nabla}
\def\om{\omega}
\def\Om{\Omega}

\def\na{\nabla}

\def\nn{\nonumber}

\def\qqquad{\quad\quad\quad}

\newtheorem{theorem}{Theorem}[section]
\newtheorem{proposition}[theorem]{Proposition}
\newtheorem{lemma}[theorem]{Lemma}

\begin{document}

\markboth{A. Zaidni, S. Benjelloun, R. Boukharfane}
{Global existence of weak solutions to incompressible anisotropic CHNS system}

%
\title{Global existence of weak solutions to incompressible anisotropic Cahn--Hilliard--Navier--Stokes system}

\author{
Azeddine Zaidni\thanks{Corresponding author: Azeddine.zaidni@um6p.ma} \\
College of computing, University Mohammed VI Polytechnic (UM6P), Morocco
\and
Saad Benjelloun \\
De Vinci Higher Education, De Vinci Research Center, Paris, France
\and
Radouan Boukharfane \\
College of computing , University Mohammed VI Polytechnic (UM6P), Morocco
}

\date{}
\maketitle
\begin{abstract}
{\bfseries Abstract.}\quad We study the anisotropic, incompressible Cahn-Hilliard-Navier-Stokes system with variable density in a bounded smooth domain 
$\Omega \subset \mathbb{R}^d$. This work extends previous results on the isotropic case by incorporating anisotropic surface energy, represented by 
$\mathfrak{F}= \int_{\Omega} \frac{\epsilon}{2}\, \Gamma^2(\nabla \phi) $. The thermodynamic consistency of this system, as well as its modeling background and physical motivation, has been established in \cite{anderson2000phase,taylor-cahn98, zaidni2024}. Using a Galerkin approximation scheme, we prove the existence of global weak solutions in both two- and three-dimensions $(d=2,3)$. A key ingredient in extending the local existence of approximate solutions to a global one is the application of Bihari's inequality combined with a fixed-point argument.
\end{abstract}


\section{Introduction}

The well-known Navier-Stokes equations govern the motion of a single-phase fluid. However, in the case of two-phase fluids, chemical reactions, phase changes, and mass transfer between phases become significant and cannot be ignored. J. W. Cahn and J. E. Hilliard were the first to derive mathematical equations that describe phase separation in a binary fluid \cite{cahn1958free}. Here we investigate generalizations that combine the Cahn-Hilliard equation with equations that describe the dynamics of fluid flow, referred to as Cahn-Hilliard-Navier-Stokes (CHNS) systems. CHNS systems are designed to describe the hydrodynamic properties of a two-phase mixture.
The phenomenon of material transport along an interface is known as the Marangoni effect. The presence of a surface tension gradient naturally induces the migration of particles, moving from regions of low tension to those of high tension. This gradient can be induced by a concentration gradient or alternatively a temperature gradient. In two-phase theory, the interface between phases is regarded as being diffuse. According to the work of Taylor \& Cahn \cite{taylor-cahn98}, one can model the diffuse interface by a single order parameter, say $\phi$, with a functional of free energy,  
\[
\mathfrak{F}= \int_{\Omega} \frac{\epsilon}{2}\, \Gamma^2(\nabla \phi) + \frac{1}{\epsilon} V(\phi)\,,
\]
where $\Gamma$ is a homogeneous function of degree one that maps a vector to a real number. More details will be provided later in this section. Here, $V$ can be any nonnegative function that equals zero at $\phi = \pm1$, and $\epsilon$ is a small parameter that approaches zero in the sharp-interface limit. We choose the order parameter $\phi$ to represent the concentration.

Anderson et al.\ \cite{anderson2000phase} proposed a phase-field solidification model with convection in which the interface may exhibit anisotropic surface energy. The choice of energy and entropy is given by
\begin{align*}
H_{\mathrm{AMW}} &= \int_{\Omega} \frac{\rho}{2}\, |\bfu|^2 + \rho u(\rho,s,\phi) 
+ \frac{\epsilon_E^2}{2}\Gamma^2(\nabla \phi)\,, \\
S_{\mathrm{AMW}} &= \int_{\Omega} \rho s - \frac{\epsilon_S^2}{2} \Gamma^2(\nabla \phi)\,,
\end{align*}
where the coefficients $\epsilon_S$ and $\epsilon_E$ are assumed to be constant, say equal to $1$. $\rho$ is the mass  density of the mixture, $\mathbf{u}$ is the mass-averaged velocity of the mixture, $s$ is the specific entropy, $\phi$  is the phase variable representing the specific concentration, and $u$ denotes the internal energy. In this case, the nonlinear system of equations governing the motion of a mixture of two incompressible non-Newtonian fluids is given by:
\begin{align}
&\partial_t \rho + \mathbf{u} \cdot \nabla \rho = 0, \label{rho} \\
&\rho \partial_t \phi + \rho \mathbf{u} \cdot \nabla \phi = \operatorname{div} (D(\phi) \nabla \mu), \label{phi} \\
&\rho \partial_t \mathbf{u} + \rho (\mathbf{u} \cdot \nabla) \mathbf{u} - \operatorname{div}(\nu(\phi) \mathbb{D}(\mathbf{u})) = -\nabla \pi - \operatorname{div}(\Gamma \mathbf{\bfxi}(\nabla\phi) \otimes \nabla \phi), \label{velocity_S} \\
&{ \operatorname{div}\bfu = 0,}\label{Eq3}
\\
&\rho \mu = \rho F'(\phi) - \operatorname{div}(\Gamma \mathbf{\bfxi}(\nabla\phi)), \label{mu}
\end{align}
where $\pi$ is the pressure, $\nu(\phi)$ is the viscosity, $D(\phi)$ is the diffusion coefficient, and $\mu$ is the chemical potential. The symmetric strain-rate tensor of $\mathbf{u}$ is denoted by $\mathbb{D}(\mathbf{u}) = \frac{1}{2}(\nabla \mathbf{u} + \nabla \mathbf{u}^{T})$. The function $F$ represents the logarithmic function defined on $[-1,1]$, which follows from a mean-field model:
\[
F(s) = \frac{\lambda_1}{2}(1- s^2) + \frac{\lambda_2}{2}\left[(1+s) \ln \frac{(1+s)}{2}+(1-s) \ln \frac{(1-s)}{2}\right] := \frac{\lambda_1}{2}(1- s^2) + G(s).\label{log_pot}
\]
where $0<\lambda_2<\lambda_1$. As above, $\Gamma$ is a homogeneous function of degree one of $\mathbf{p} = (p_1, \ldots, p_d) \in \mathbb{R}^d$, and $\bfxi$ is defined as:
\[
\Gamma(\mathbf{p}) = \mathbf{p} \cdot \bfxi \coloneqq p_j \frac{\partial \Gamma(\mathbf{p})}{\partial p_j}.
\]
Differentiating this gives:
\[
\frac{\partial \Gamma}{\partial p_i} = \frac{\partial}{\partial p_i} ( \mathbf{p}\cdot \bfxi) = \frac{\partial^2 \Gamma}{\partial p_i \partial p_j} p_j + \xi_i = \xi_i,
\]
where $p_j$ must be a null eigenvector of the matrix $\partial^2 \Gamma/\partial p_i \partial p_j$. Henceforth, we will assume the argument of $\Gamma$ to be $\nabla \phi$. For the case of isotropic surface energy, where $\Gamma(\nabla \phi) = |\nabla \phi|$, the associated homogeneous function of degree zero is given by $\bfxi = \nabla \phi / |\nabla \phi|$. The thermodynamic consistency for this system, for any anisotropic surface energy, was established in \cite{anderson2000phase, zaidni2024}, where the reader is also referred to for further details on the modeling background. The system is considered in $\Omega \times (0, T)$, where $\Omega$ is a bounded domain (open and connected set) in $\mathbb{R}^d$ for $d = 2,3$, with a regular boundary $\partial \Omega$, and $T > 0$ is a given positive time.

We complete the system with the following boundary and initial conditions:
\[
\begin{cases}
\mathbf{u} = 0,~\partial_{\mathbf{n}} \mu = \partial_{\mathbf{n}} \phi = 0,~\Gamma \bfxi(\nabla \phi) \cdot \mathbf{n} = 0 & \text{on}~~\partial \Omega \times (0, T), \\
\rho(\cdot, 0) = \rho_0,~\mathbf{u}(\cdot, 0) = \mathbf{u}_0,~\phi(\cdot, 0) = \phi_0 & \text{in}~~\Omega.
\end{cases}\label{Boundray_conditionss}
\]
Throughout this work, the functions $\nu(s)$ and $D(s)$ are assumed to be in \( W^{1,\infty}(\mathbb{R}) \) such that \( 0 < \nu_* \leq \nu \leq \nu^* \) and \( 0 < D_* \leq D \leq D^* \), where $\nu_*$, $\nu^*$, $D_*$, and $D^*$ are positive constants.

Previous works mainly focused on the isotropic case. {We present below a concise review of relevant results on systems with variable density}. The existence of global weak solutions for the incompressible Cahn-Hilliard-Navier-Stokes (CHNS) system with variable density has been established in both two and three dimensions \cite{giorgini2020weak,abels2024,munteanu2024,rui2024global,abels2013}. For results on strong solutions, we refer the reader to { \cite{giorgini2020weak, Zhao, li2024strong, kotschote2015strong}}. In the compressible case, results have been established in one space dimension \cite{Chen,Cherfils,elbar2024,giorgini2021navier,dingwell}. Beyond the isotropic case, the global existence of solutions for higher-order CHNS systems in two dimensions has also been established \cite{Jiaojiao}. {Some recent works have addressed anisotropic effects, but only in the context of the Cahn–Hilliard equation. In particular, in \cite{Garcke, garcke2023}, the authors establish the existence and uniqueness results for weak solutions in the anisotropic setting. In addition, a variant of this model, consisting of the Navier–Stokes–Voigt equations \cite{dilmi2025} coupled with the Cahn–Hilliard equation, has been investigated in \cite{giorgini2023existence}. In a related direction, compressible Navier--Stokes systems with non-monotone pressure have been shown to generate oscillations and phase-transition effects, offering an alternative description closely connected to diffuse-interface models such as CHNS \cite{tzavaras2024}.
}\\

In this paper, we prove the existence of global weak solutions for the incompressible anisotropic Cahn-Hilliard-Navier-Stokes system \eqref{rho}-\eqref{mu}. To this end, we state our assumptions on the anisotropic surface energy. We assume that the function $\Ga$ satisfies the following:
\begin{itemize}
\item $\textbf{H}_1$: There exist two positive numbers, $r$ and $R$, such that for any $\mathbf{p} \in \mathbb{R}^d$, we have \( r\|\mathbf{p}\|_2^2 \leq \Gamma^2(\mathbf{p}) \leq R\|\mathbf{p}\|_2^2 \), where \( \|\mathbf{p}\|_2^2 = \sum_{i=1}^d p_i^2 \).
\item $\textbf{H}_2$: The map \( \mathbf{p} \mapsto \Ga {\bfxi}(\mathbf{p}) \) is linear.
\item $\textbf{H}_3$: { 
$\Ga {\bfxi}$
is assumed to be strongly monotone, i.e., there
exists a constant $a_\Ga > 0$ such that
\[
\Ga {\bfxi}(\mathbf{p})\cdot  \mathbf{p} \geq a_\Ga |\bfp|^2.
\]
}
\end{itemize}
Taylor \& Cahn \cite{taylor-cahn98} provide a family of candidate functions $\Gamma$, specifically:
\[
\Gamma_{\alpha,\beta}^2 (\mathbf{p}) = p_1^2 + p_2^2 + p_3^2 + 2\alpha(|p_1 p_2| + |p_1 p_3| + |p_2 p_3|) + 2\beta(|p_1 - p_2|^2 + |p_1 - p_3|^2 + |p_2 - p_3|^2),
\]
for $d = 3$ and $\alpha, \beta > -1$. The functions $\Gamma_{\alpha,\beta}$ satisfy $\textbf{H}_1$, $\textbf{H}_2$, and $\textbf{H}_3$ for $ \alpha=0$ and $\beta > {-1}/{8}$. A set of functions that satisfy the previous assumptions is provided by $\Gamma^2(\mathbf{p}) = \mathbf{p}^T \mathbf{M} \mathbf{p}$, where $\mathbf{M}$ is a positive definite matrix. In this case, $R$ and $r$ represent the largest and smallest eigenvalues of the matrix $\mathbf{M}$, respectively.

Here, we give an example where $\al=0$ and $\beta = 1/2$. In this case, we have 
\bq
\resizebox{1\hsize}{!}{
$
\Ga\bfxi (\nabla\phi)\otimes\nabla\phi = \left[\begin{array}{ccc}
3(\partial_x\phi)^2 -\partial_x\phi\partial_y\phi-\partial_x\phi\partial_z\phi& -(\partial_y\phi)^2 + 3\partial_x\phi\partial_y\phi-\partial_y\phi\partial_z\phi & -(\partial_z\phi)^2 -\partial_y\phi\partial_z\phi+3\partial_x\phi\partial_z\phi \\
-(\partial_x\phi)^2 +3\partial_x\phi\partial_y\phi-\partial_x\phi\partial_z\phi & 3(\partial_y\phi)^2 -\partial_x\phi\partial_y\phi-\partial_y\phi\partial_z\phi & -(\partial_z\phi)^2 +3\partial_y\phi\partial_z\phi-\partial_y\phi\partial_z\phi \\
-(\partial_x\phi)^2 -\partial_x\phi\partial_y\phi+3\partial_x\phi\partial_z\phi & -(\partial_y\phi)^2 -\partial_x\phi\partial_y\phi+3\partial_y\phi\partial_z\phi & 3(\partial_z\phi)^2 -\partial_z\phi\partial_x\phi-\partial_z\phi\partial_y\phi
\end{array}\right]
$}
\eq
The equation \eqref{mu} then becomes

\bq\rho\mu = \rho F^{\prime}(\phi) - 3 \Delta \phi - 2 \partial_{xy}\phi - 2 \partial_{xz}\phi - 2 
\partial_{yz}\phi.\eq

\section{Main Result}
Before formulating our main result, we need to introduce some definitions and notation. Let $\mathcal{C}_c^{k}(\Omega),\,\, k\in \mathbb{N}$ be the space of all functions that, together with all partial derivatives up to order $k$, are continuous in $\Omega$. Let $p \geq 1$ and $q > 0$. Then $L^p (\Omega)$ and $W^{q,p} (\Omega)$ are the usual Lebesgue and Sobolev spaces, respectively. $(f, g):= \int_{\Omega} f(x) g(x) \, dx$ denotes the scalar product with respect to the spatial variable. Moreover, for a Banach space $X$ and a real interval $I$, we
denote by $L^p(I; X)$ the Bochner space, which is equipped with the norm $\left(\int_I \|\,.\,\|_X^p \, dt \right)^{1/p}$.

We introduce the following velocity spaces $$\mathfrak{C}= \left\{\mathbf{u} \in \left(\mathcal{C}_c^{\infty}(\Omega)\right)^d \, ; \, \operatorname{div}\mathbf{u} = 0\right\},$$
$$\mathbf{V} = \text{the closure of } \mathfrak{C} \text{ in } \left(H^1_0(\Omega)\right)^d,$$
$$\mathbf{H} = \text{the closure of } \mathfrak{C} \text{ in } \left(L^2(\Omega)\right)^d,$$
where $\mathcal{C}_c^{\infty}(\Omega)$ is the space of smooth functions with compact support in $\Omega$. 

Given $1 \leq p \leq +\infty$, the Besov spaces denoted by $B_{p, \infty}^{\frac{1}{4}}(0,T ; X)$ consist of the set of functions $f \in L^p(0,T ; X)$ with finite norm
\begin{equation*}
\|f\|_{B_{p, \infty}^{\frac{1}{4}}(0,T ; X)} = \|f\|_{L^p(0,T ; X)} + \sup_{0<h \leq 1} h^{-\frac{1}{4}}\left\|\Delta_h f\right\|_{L^p\left(I_h ; X\right)},
\end{equation*}
where $\Delta_h f(t) = f(t+h) - f(t)$ and $I_h = \{t \in (0,T): t+h \in (0,T)\}$.

We now state the main result of our paper as follows.

\begin{theorem} \label{th1} 
Let $T$ be a positive time. Assume that $\rho_0 \in L^\infty(\Omega)$ with $0 < \rho_* \leq \rho_0 \leq \rho^* < \infty$, $\mathbf{u}_0 \in L^2(\Omega)$, {  whose divergence vanishes in the weak sense}, and $\phi_0 \in H^1(\Omega) \cap L^{\infty}(\Omega)$ such that $\|\phi_0\|_{L^{\infty}(\Omega)} \leq 1$ and  { 
$-1 < \frac{\overline{\rho_0 \phi_0}}{\rho^*} < 1$, where $\overline{\rho_0\phi_0} = \frac{1}{|\Om|}\int_{\Om}\rho_0(x)\phi_0(x)\, \md x$}. 

Then, there exists a weak solution $(\rho, \mathbf{u}, \phi, \mu)$ to the system \eqref{rho}-\eqref{mu} satisfying
\begin{enumerate}
\item 
\bal
& \rho \in \mathcal{C}([0,T];L^2(\Omega)) \cap L^{\infty}(\Omega \times (0, T)) \cap W^{1, \infty}\left(0, T ; H^{-1}(\Omega)\right), \label{reg_rho}\\
& \mathbf{u} \in L^2\left(0, T ; \mathbf{V}\right) \cap B_{2, \infty}^{\frac{1}{4}}\left(0, T ; \mathbf{H}\right), \\
& \phi \in  { L^{\infty}\left(0, T ; H^1(\Omega)\right)} \cap B_{\infty, \infty}^{\frac{1}{4}}\left(0, T ; L^2(\Omega)\right), \\
& \mu \in 
{ L^2\left(0, T ; H^1(\Omega)\right)},\label{reg_nu}
\eal
{ 
\item 
\begin{equation}
\int_0^T \!\!\int_{\Omega} \rho\,\partial_t \psi \, \md x\, \md t
+
\int_0^T \!\!\int_{\Omega} \rho\, \bfu \cdot \nabla \psi \, \md x\, \md t
= 0,
\end{equation}
for all $\psi \in C_c^\infty((0,T)\times \Omega )$.

\item \begin{align}
&- \int_0^T \!\!\int_{\Omega} \rho\, \bfu \cdot \partial_t \bfv \, \md x\, \md t
- \int_0^T \!\!\int_{\Omega} \rho\, \bfu \otimes \bfu : \nabla \bfv \, \md x\, \md t
+ \int_0^T \!\!\int_{\Omega} \nu(\phi)\, \mathbb{D} (\bfu) : \nabla \bfv \, \md x\, \md t \notag \\
&\qquad
= \int_{\Omega} \rho_0 \bfu_0 \cdot \bfv(0) \, \md x
+ \int_0^T \!\!\int_{\Omega} \Ga \bfxi(\nabla \phi) \otimes \nabla \phi : \nabla \bfv \, \md x\, \md t ,
\end{align}
for all $\bfv \in C_c^1\!\left([0,T); \bfV\right)$.

\item \begin{align}
&- \int_0^T \!\!\int_{\Omega} \rho\, \partial_t \phi \, w \, \md x\, \md t
- \int_0^T \!\!\int_{\Omega} \rho\, \phi\, \bfu  \cdot \nabla w \, \md x\, \md t
+ \int_0^T \!\!\int_{\Omega} \nabla \mu \cdot \nabla w \, \md x\, \md t \notag \\
&\qquad
= \int_{\Omega} \rho_0 \phi_0 \, w(0) \, \md x ,
\end{align}
for all $w \in C_c^1\!\left([0,T); H^1(\Omega)\right)$.

}

\item Additionally, we have
\bq
\rho \mu = \rho F^{\prime} (\phi) - \operatorname{div}(\Gamma \mathbf{\bfxi}(\nabla\phi) ) \quad \text{a.e. in } (0,T) \times \Omega.
\eq
\end{enumerate}
\end{theorem}

The proof of Theorem \ref{th1} is organized as follows. In Section \S\ref{Energy balance}, we derive the energy balance. In Section \S\ref{Reg_log}, we introduce a system with a regularized logarithmic potential, which is defined on $\mathbb{R}$ and parameterized by $\varepsilon \in \left(0,1-\sqrt{1-{\lambda_2}/{\lambda_1}}\right)$. In Section \S\ref{Galerkin scheme}, we establish, through a Galerkin scheme, the global existence of the solution to the system with the regularized logarithmic potential. The proof is divided into three steps:
\begin{itemize}
\item \textbf{Step 1}: We construct a linearized version of the system and prove the existence of a solution using the classical Cauchy-Lipschitz theorem. The solution is defined only on a local time interval $[0,T_0)$, where $T_0$ depends on the initial data.
\item \textbf{Step 2}: By analyzing the map that takes the function around which we linearize as input and returns the solution given by the Cauchy–Lipschitz theorem as output, and applying Bihari's inequality \cite{bihari1956generalization}, we establish the existence of a small time interval $[0,\widetilde{T}]$ where both the input and output data remain well-defined.

\item \textbf{Step 3}: We establish the regularity of the solution and we show that the map from input data to output data has a Schauder fixed point. This argument follows the same approach used in \cite[see Appendix A]{giorgini2020weak}.
\end{itemize}
At the end of Section \S\ref{Galerkin scheme}, we show that the fixed point cannot blow-up up in finite time $T_0$. In Section \S\ref{Global existence}, we present the proof of the main result, which is structured into two parts. First, we establish that the Galerkin approximation has a convergent subsequence, and its limit is a global weak solution of the system with the regularized logarithmic potential. Second, we prove uniform estimates, and by passing to the limit as $\varepsilon \rightarrow 0^{+}$, we obtain a weak solution of our system \eqref{rho}-\eqref{mu}.

\section{Energy balance}\label{Energy balance}
In this section, we derive the total energy balance. By multiplying \eqref{velocity_S} by $\bfu$ and integrating over $\Omega$, we obtain
\bq
\frac{1}{2}\frac{ \md}{ \md t}\int_{\Omega} \rho |\bfu|^2 \md x + \int_{\Omega}\nu(\phi)|\mathbb{D}(\bfu)|^2 \md x = \int_{\Omega} - \operatorname{div}({ \Ga\bfxi(\na \phi)}\otimes\nabla\phi)\cdot\bfu\,\md x.\label{Energy_K}
\eq
We have
\bal
\operatorname{div} ({ \Ga\bfxi(\nabla\phi)}\otimes\nabla\phi) &= \nabla\phi(\operatorname{div}({ \Ga\bfxi(\nabla\phi)})) + { { \Ga\bfxi(\nabla\phi)}\cdot\nabla(\nabla\phi)} \nn\\
&= \nabla\phi(\operatorname{div}({ \Ga\bfxi(\nabla\phi)})) + \nabla\left(\frac{1}{2}\Gamma^2(\nabla\phi)\right).
\eal
Using \eqref{mu}, we find
\bal
-\operatorname{div} ({ \Ga\bfxi(\nabla\phi)}\otimes\nabla\phi) &=
\rho \mu \nabla\phi - \rho F'(\phi) \nabla \phi - \nabla\left(\frac{1}{2}\Gamma^2(\nabla\phi)\right)\nn\\
&=\rho \mu \nabla\phi - \rho \nabla F(\phi) - \nabla\left(\frac{1}{2}\Gamma^2(\nabla\phi)\right).
\eal
Assuming the boundary condition on $\bfu$ \eqref{Boundray_conditionss}, the equation \eqref{Energy_K} becomes 
\bq
\begin{aligned}
\frac{1}{2}\frac{ \md}{ \md t}\int_{\Omega} \rho |\bfu|^2 \md x &+ \int_{\Omega}\nu(\phi)|\mathbb{D}(\bfu)|^2 \md x =\int_{\Omega} \rho \mu \bfu\cdot\nabla\phi- \rho \bfu\cdot\nabla F(\phi)\md x.\label{Kin}
\end{aligned}
\eq
Multiplying \eqref{mu} by $\partial_t \phi$ and integrating, we have
\bq
\begin{aligned}
\int_{\Omega} \rho\mu \, \partial_t\phi =& - \int_{\Omega} \partial_t \phi \, \operatorname{div}({ \Ga\bfxi(\nabla\phi)})\md x+ \int_{\Omega}\rho F^{\prime}(\phi)\partial_t\phi\, \md x \\
=& \int_{\Omega} { \Ga\bfxi(\nabla\phi)} \cdot \partial_t\nabla\phi \md x + \int_{\Omega} \rho F^{\prime}(\phi) \partial_t \phi\, \md x\\
=&\frac{\md}{\md t}\int_{\Omega} \frac{1}{2}\Gamma^2(\nabla\phi) + \rho F(\phi) \, \md x -\int_{\Omega} F(\phi)\partial_t\rho\, \md x.
\label{intern}
\end{aligned}
\eq
Here, we used the boundary conditions $\Ga\bfxi(\nabla\phi)\cdot\bfn = 0$.

Multiplying \eqref{phi} by $\mu$ and integrating, we find using the Neumann boundary conditions on $\mu$
\bq
\int_{\Omega} \rho \mu \partial_t \phi \, \md x + \int_{\Omega} \rho \mu \bfu \cdot \nabla \phi\, \md x + \int_{\Omega} D(\phi)|\nabla\mu|^2 \md x = 0. \label{diff} 
\eq
By combining \eqref{Kin}, \eqref{intern}, and \eqref{diff}, we get
\bq
\begin{aligned}
\frac{ \md}{ \md t}\int_{\Omega} \frac{1}{2}\rho |\bfu|^2 + \frac{1}{2}\Gamma^2(\nabla\phi) + \rho F(\phi)\,\md x + \int_{\Omega}\nu(\phi)|\mathbb{D}(\bfu)|^2 + D(\phi)|\nabla \mu|^2 \md x= 0.
\end{aligned}
\eq

Here, the Ginzburg–Landau energy $\int_{\Omega} \frac{1}{2}\Gamma^2(\nabla\phi) + \rho F(\phi)\,\md x $ characterizes the interfacial energy in the region where $\phi$ departs from the minima of $F(\phi)$ and $\int_{\Omega} \frac{1}{2}\rho |\bfu|^2\,\md x$ is the kinetic energy of the fluid.

\section{Regularized logarithmic potential}\label{Reg_log}

We introduce a family of regular potentials $F_{\varepsilon}$ approximating the logarithmic potential \eqref{log_pot}. For any $\varepsilon \in (0,1 - \sqrt{1 - {\lambda_2}/{\lambda_1}})$, we define
\begin{equation}
F_{\varepsilon}(s) = \frac{\lambda_1}{2}(1 - s^2) + G_{\varepsilon}(s),
\end{equation}
where $G_{\varepsilon}$ is defined by
\begin{equation}
G_{\varepsilon}(s)= 
\begin{cases} 
\sum_{j=0}^2 \frac{1}{j!} G^{(j)}(1-\varepsilon)[s-(1-\varepsilon)]^j & \text{for } s \geq 1-\varepsilon, \\ 
G(s) & \text{for } s \in [-1+\varepsilon, 1-\varepsilon], \\ 
\sum_{j=0}^2 \frac{1}{j!} G^{(j)}(-1+\varepsilon)[s-(-1+\varepsilon)]^j & \text{for } s \leq -1+\varepsilon.
\end{cases}
\end{equation}
We have $F_{\varepsilon} \in \mathcal{C}^2(\mathbb{R})$ and 
\begin{equation}
F_{\varepsilon}(s) \leq F(s) \quad \text{for all } s \in (-1,1), \quad
|F_{\varepsilon}^{\prime}(s)| \leq |F^{\prime}(s)| \quad \text{for all } s \in (-1,1). \label{F_propriety}
\end{equation}

In order to construct weak solutions by an approximation procedure, we introduce the following regularized problem:
\begin{align}
&\partial_t \rho^{\varepsilon} + \mathbf{u}^{\varepsilon} \cdot \nabla \rho^{\varepsilon} = 0, \label{rho_epsilon} \\
&\rho^{\varepsilon} \partial_t \phi^{\varepsilon} + \rho^{\varepsilon} \mathbf{u}^{\varepsilon} \cdot \nabla \phi^{\varepsilon} = \operatorname{div} (D(\phi^{\varepsilon}) \nabla \mu^{\varepsilon}), \label{phi_epsilon} \\
&\rho^{\varepsilon} \partial_t \mathbf{u}^{\varepsilon} + \rho^{\varepsilon} (\mathbf{u}^{\varepsilon} \cdot \nabla^{\varepsilon}) \mathbf{u}^{\varepsilon} - \operatorname{div}(\nu(\phi^{\varepsilon}) \mathbb{D}(\mathbf{u}^{\varepsilon})) = -\nabla \pi^{\varepsilon} - \operatorname{div}(\Gamma \mathbf{\bfxi}(\nabla \phi^{\varepsilon}) \otimes \nabla \phi^{\varepsilon})\label{velocity_S_epsilon}, \\
&{ \operatorname{div}\bfu^{\varepsilon} = 0\,,}
\\
&\rho^{\varepsilon} \mu^{\varepsilon} = \rho^{\varepsilon} F^{\prime}_{\varepsilon} (\phi^{\varepsilon}) - \operatorname{div}(\Gamma \mathbf{\bfxi}(\nabla \phi^{\varepsilon})). \label{muepsilon}
\end{align}
Furthermore, we complete the system with the same boundary and initial conditions \eqref{Boundray_conditionss}. The existence of weak solutions to the system \eqref{rho_epsilon}--\eqref{muepsilon} can be established by a standard argument: in the case of bounded domains, we construct approximate solutions via a Galerkin approximation scheme, derive uniform bounds, and thus obtain solutions by passing to the limit. This will be detailed in section \ref{Galerkin scheme}.

\section{Global existence of an approximate solution}\label{Galerkin scheme}

In this section, we construct a converging sequence via a Galerkin approximation scheme. We define the inner product on $\mathbf{V}$ by $(\mathbf{u},\mathbf{v})_{\mathbf{V}} = (\nabla\mathbf{u},\nabla\mathbf{v})$ and the norm $\|\mathbf{u}\|_{\mathbf{V}} = \|\nabla\mathbf{u}\|_{L^2(\Omega)}$. For any integer $n \geq 1$, we define the finite-dimensional subspaces of $\mathbf{V}$ and $H^1(\Omega)$, respectively, by $\mathbf{V}_n = \text{span}\{\mathbf{v}_1,\ldots,\mathbf{v}_n\}$ and $H_n = \text{span}\{\omega_1,\ldots,\omega_n\}$, where the families of functions $\{\mathbf{v}_j\}_{j\geq 1}$ and $\{\omega_j\}_{j\geq 1}$ represent the eigenfunctions of the Stokes operator (with zero Dirichlet boundary conditions) and the Laplace operator (with zero Neumann boundary conditions and $\Gamma \bfxi(\nabla\phi)\cdot\bfn = 0$), respectively. We denote $(0 < \lambda_1^s \leq \lambda_2^s \leq \ldots)$ and $(1 = \lambda_1 \leq \lambda_2 \leq \ldots)$ as the corresponding eigenvalues of the Stokes operator and the Laplace operator, respectively. We denote by $P_n^1$ and $P_n^2$ the orthogonal projections onto $\mathbf{V}_n$ and $H_n$ with respect to the inner product in $\mathbf{H}$ and $L^2(\Omega)$, respectively. We define the initial data as follows: we consider the triplet $(\rho_{0n},\mathbf{u}_{0n},\phi_{0n})$, where $\mathbf{u}_{0n} = P_n^1\mathbf{u}_0$ and $\phi_{0n} = P_n^{2}\phi_0$, and $\rho_{0n}$ constructed using mollification by convolution such that for any $n \geq 1$, $\rho_{0n} \in \mathcal{C}^{\infty}(\overline{\Omega})$ and $\rho_{*} \leq \rho_{0n}(x) \leq \rho^{*}$ for any $x \in \overline{\Omega}$. The constructed initial data satisfy
\begin{align}
\rho_{0n} &\rightarrow \rho_0 \quad \text{strongly in } L^2(\Omega), \\
\rho_{0n} &\rightharpoonup \rho_0 \quad \text{weak-star in } L^{\infty}(\Omega), \\
\mathbf{u}_{0n} &\rightarrow \mathbf{u}_0 \quad \text{strongly in } (L^2(\Omega))^d, \\
\phi_{0n} &\rightarrow \phi_0 \quad \text{strongly in } H^1(\Omega). \label{phi_0m}
\end{align}
We have $\bfu_{0n} = 0, \quad\partial_{\bfn}\mu_{0n} = \partial_{\bfn}\phi_{0n} = 0, \text{ on } \partial\Omega\times(0,{T})$. Since $\Ga\bfxi$ is linear, we also have
$$\Ga\bfxi(\nabla\phi_{0n})\cdot\bfn = \sum_{i=1}^n(\phi_{0n},\omega_i)_{L^2(\Omega)} \Ga\bfxi(\nabla\omega_i)\cdot\bfn=0.$$
We recall that for any $\bfu_n \in \bfV_n$ and $\phi_n\in H_n$, we have
\begin{align}
\|\mathbf{u}_n\|_{L^2(\Omega)} &\leq \frac{1}{\sqrt{\lambda_1^s}} \|\nabla\mathbf{u}\|_{L^2(\Omega)}, \label{poincare_ineq} \\
\|\mathbf{u}_n\|_{\mathbf{V}} &\leq \sqrt{\lambda_n^s}\|\mathbf{u}_n\|_{L^2(\Omega)}, \\
\|\phi_n\|_{H^2(\Omega)} &\leq \sqrt{\lambda_n}\|\phi_n\|_{H^1(\Omega)}, \\
\|\mathbf{u}_n\|_{L^{\infty}(\Omega)} &\leq K\|\mathbf{u}_n\|_{H^2(\Omega)} \leq K_n \|\mathbf{u}_n\|_{L^2(\Omega)},\label{u_n_inq}
\end{align}
where $K$ and $K_n$ are constants. Let $1\leq p\leq +\infty$.

\begin{proposition}{(Local existence of an approximate solution)}. Given the initial data $(\rho_{0n},\bfu_{0n},\phi_{0n},\mu_{0n})$ constructed as above, there exists a time interval $[0,\widetilde{T}]$ with $\widetilde{T} > 0$ and $(\rho_n^{\varepsilon},\bfu_n^{\varepsilon},\phi_n^{\varepsilon},\mu_n^{\varepsilon})$ such that
\bq
\rho_n^{\varepsilon} \in \mathcal{C}^1(\overline{\Omega}\times[0,\widetilde{T}]),\,\,\bfu_n^{\varepsilon}\in \mathcal{C}^1([0,\widetilde{T}];\bfV_n),\,\, \phi_n^{\varepsilon}\in \mathcal{C}^1([0,\widetilde{T}];H_n), \,\, \mu_n^{\varepsilon}\in \mathcal{C}([0,\widetilde{T}];H_n),
\eq
\bq
\partial_t\rho_n^{\varepsilon} + \bfu_n^{\varepsilon}\cdot\nabla\rho_n^{\varepsilon} = 0 \text{ in } \Omega\times(0,\widetilde{T}),\label{density_T}
\eq
\bq
\begin{aligned}
(\rho_n^{\varepsilon}\partial_t \bfu_n^{\varepsilon},\bfv) + (\rho_n^{\varepsilon}(\bfu_n^{\varepsilon}\cdot\nabla)\bfu_n^{\varepsilon},\bfv) + (\nu(\phi_n^{\varepsilon})\mathbb{D}\bfu_n^{\varepsilon},\nabla\bfv)=& (\rho_n^{\varepsilon}\mu_n^{\varepsilon}\nabla\phi_n^{\varepsilon},\bfv)\\
&-(\rho_n^{\varepsilon}\nabla(F_{\varepsilon}(\phi_n^{\varepsilon})),\bfv),\label{velocity_T}
\end{aligned}
\eq
\bq
(\rho_n^{\varepsilon}\partial_t\phi_n^{\varepsilon},\omega) + (\rho_n^{\varepsilon}\bfu_n^{\varepsilon}\cdot\nabla\phi_n^{\varepsilon},\omega)+ (D(\phi_n^{\varepsilon})\nabla\mu_n^{\varepsilon},\nabla\omega) = 0,\label{concentration_T}
\eq
\bq
(\rho_n^{\varepsilon}\mu_n^{\varepsilon},\omega) =(\Ga\bfxi(\nabla\phi_n^{\varepsilon}),\nabla\omega) + (\rho_n^{\varepsilon} F_{\varepsilon}^{\prime}(\phi_n^{\varepsilon}),\omega),\label{chemical_T}
\eq
for all $\bfv\in \bfV_n$, $\omega\in H_n$ and for all $t\in [0,\widetilde{T}]$.
\end{proposition}

\begin{proof}
\textbf{Step 1}: For any $n\geq 1$, let $(\widetilde{\bfu}^{\varepsilon}_n,\widetilde{\phi}_n^{\varepsilon})$ be such that $\widetilde{\bfu}_n^{ \varepsilon} \in \mathcal{C}([0,T];\bfV_n)$ 
and $\quad \widetilde{\phi}_n^{\varepsilon}\in \mathcal{C}([0,T];H_n).$ The density corresponding to the given velocity $\widetilde{\bfu}_n^{\varepsilon}$ is then explicitly given by  
\bq
\rho_n^{\varepsilon}(x,t) = \rho_{0n}(\widetilde{y}^{\varepsilon}_n(0,t,x)),\quad \rho_n^{\varepsilon} \in \mathcal{C}^1 (\overline{\Omega}\times[0,T]),\label{regularity-density}
\eq
where $\widetilde{y}^{\varepsilon}_n$ is the unique solution of the following Cauchy problem
\bq
\frac{\md \widetilde{y}^{\varepsilon}_n}{\md \tau}(\tau,t,x) = \widetilde{\bfu}^{\varepsilon}_n(\widetilde{y}^{\varepsilon}_n(\tau,t,x),\tau),\quad\quad \widetilde{y}^{\varepsilon}_n(t,t,x) = x.
\eq
We have the following estimates 
\bq
\rho_*\leq \rho_n^{\varepsilon}(x,t)\leq\rho^* \text{ for any } (x,t)\in \overline{\Omega}\times[0,T],
\eq
\bq
\max_{t\in [0,T]}\|\nabla\rho_n^{\varepsilon}(t)\|_{L^{\infty}(\Omega)}\leq K \|\nabla\rho_{0n}\|_{L^{\infty}(\Omega)}\exp\left(\int_0^T \|\widetilde{\bfu}^{\varepsilon}_n(\tau)\|_{W^{1,\infty}(\Omega)}\md \tau\right),\label{estim_density}
\eq
where the constant $K$ is independent of the integer $n$. Given a triplet $(\bfu_n,\phi_n,\mu_n)$ such that

$$\bfu_n^{\varepsilon}(x,t) = \sum_{i=1}^n {\alpha}_i^{\varepsilon}(t) \bfv_i(x),\quad \phi_n^{\varepsilon}(x,t) = \sum_{i=1}^n {\beta}^{\varepsilon}_i(t) \omega_i(x),\quad \mu_n^{\varepsilon}(x,t) = \sum_{i=1}^n {\gamma}_i^{\varepsilon}(t) \omega_i(x),$$

which solves the following system
\bal
(\rho_n^{\varepsilon}\partial_t \bfu_n^{\varepsilon},\bfv_i) + (\rho_n^{\varepsilon}(\widetilde{\bfu}_n^{\varepsilon}\cdot\nabla)\bfu_n^{\varepsilon},\bfv_i) + (\nu(\widetilde{\phi}_n^{\varepsilon})\mathbb{D}\bfu_n^{\varepsilon},\nabla\bfv_i)=&(\rho_n^{\varepsilon}\mu_n^{\varepsilon}\nabla\widetilde{\phi}_n^{\varepsilon},\bfv_i)\nn\\
&-(\rho_n^{\varepsilon}\nabla(F_{\varepsilon}(\phi_n^{\varepsilon})),\bfv_i),
\label{velocity}
\eal
\bq
(\rho_n^{\varepsilon}\partial_t\phi_n^{\varepsilon},\omega_i) + (\rho_n^{\varepsilon}\bfu_n^{\varepsilon}\cdot\nabla\widetilde{\phi}_n^{\varepsilon},\omega_i)+ (D(\widetilde{\phi}_n^{\varepsilon})\nabla\mu_n^{\varepsilon},\nabla\omega_i) = 0,\label{concentration}
\eq

\bq
(\rho_n^{\varepsilon}\mu_n^{\varepsilon},\omega_i) = (\Ga\bfxi(\nabla\phi_n^{\varepsilon}),\nabla\omega_i) + (\rho_n^{\varepsilon} F_{\varepsilon}^{\prime}(\phi_n^{\varepsilon}),\omega_i),\label{chemical}
\eq
for all $i=1,\ldots,N$. 

We define the unknown multi-component quantities as follows
$\mathbf{A}^n = (\alpha_1^{\varepsilon},\ldots,\alpha_n^{\varepsilon})$, $\mathbf{B}^n = (\beta_1^{\varepsilon},\ldots,\beta_n^{\varepsilon})$ and $\mathbf{C}^n = (\gamma_1^{\varepsilon},\ldots,\gamma_n^{\varepsilon})$. Hence, from the classical Cauchy-Lipschitz theorem, we obtain the existence and uniqueness of a maximal solution
$\left(\mathbf{A}^n, \mathbf{B}^n\right) \in \mathcal{C}^1\left(\left[0, T_0\right); \mathbb{R}^n \times \mathbb{R}^n\right), \mathbf{C}^n \in$ $\mathcal{C}\left(\left[0, T_0\right) ; \mathbb{R}^n\right)$.

\textbf{Step 2:} Multiplying \eqref{velocity} by $\alpha_i^{\varepsilon}$ and summing over $i$, we find
\bq
\begin{aligned}
\frac{ \md}{ \md t}\int_{\Omega} \frac{1}{2}\rho_n^{\varepsilon} |\bfu_n^{\varepsilon}|^2 \md x + \int_{\Omega}\nu(\widetilde{\phi}_n^{\varepsilon})|\mathbb{D}(\bfu_n^{\varepsilon})|^2 \md x=\int_{\Omega}& \rho_n^{\varepsilon} \mu_n^{\varepsilon} \bfu_n^{\varepsilon}\cdot\nabla\widetilde{\phi}_n^{\varepsilon}\md x \\
&-\int_{\Omega} \rho_n^{\varepsilon} \bfu_n\cdot\nabla F_{\varepsilon}(\phi_n^{\varepsilon})\md x.\label{KIN2}
\end{aligned}
\eq
Here, the density is the solution to the transport equation with velocity $\widetilde{\bfu}^{\varepsilon}_n$. Now, we multiply
\eqref{concentration} by $\gamma_i^{\varepsilon}$ and \eqref{chemical} by $\frac{\md }{\md t} \beta_i^{\varepsilon}$, respectively, and summing over $i$, we get
\bq 
\begin{aligned}
\frac{\md }{\md t} \int_{\Omega} \left[\frac{1}{2}\Ga^2(\nabla\phi_n^{\varepsilon}) + \rho_n^{\varepsilon} F(\phi_n^{\varepsilon})\right] \md x +\int_{\Omega}D(\widetilde{\phi}_n^{\varepsilon})|\nabla\mu_n^{\varepsilon}|^2 \md x &+ \int_{\Omega}\rho_n^{\varepsilon} \bfu_n^{\varepsilon}\cdot\nabla\widetilde{\phi}_n^{\varepsilon}\mu_n^{\varepsilon}\md x \\
&= \int_{\Omega} \partial_t \rho_n^{\varepsilon} F_{\varepsilon}(\phi_n^{\varepsilon})\md x\label{Pot}
\end{aligned}
\eq
By adding the two previous equations \eqref{KIN2} and \eqref{Pot}, we obtain 

\bq
\frac{\md }{\md t} E^{\varepsilon}(\rho_n^{\varepsilon},\bfu_n^{\varepsilon},\phi_n^{\varepsilon}) + \int_{\Omega} \nu(\widetilde{\phi}_n^{\varepsilon})|\mathbb{D}\bfu_n^{\varepsilon}|^2 + D(\widetilde{\phi}_n^{\varepsilon})|\nabla \mu_n^{\varepsilon}|^2 = \mathcal{R}^{\varepsilon},\label{Energy_Rest}
\eq
where $E^{\varepsilon}$ and $\mathcal{R}^{\varepsilon}$ are given by
\bq
E^{\varepsilon}(\rho_n^{\varepsilon},\bfu_n^{\varepsilon},\phi_n^{\varepsilon}) = \int_{\Omega} \frac{1}{2} \rho_n^{\varepsilon} |u_n^{\varepsilon}|^2 + \frac{1}{2}\Ga^2(\nabla\phi_n^{\varepsilon}) + \rho_n^{\varepsilon} F(\phi_n^{\varepsilon}) \,\md x,\eq
\bq
\mathcal{R}^{\varepsilon} = \int_{\Omega} F_{\varepsilon}(\phi_n^{\varepsilon})(\bfu_n^{\varepsilon} - \widetilde{\bfu}_n^{\varepsilon})\cdot\nabla\rho_n^{\varepsilon} \,\md x.
\eq
We note that $\mathcal{R}^{\varepsilon}$ will be zero at the fixed point. 

{  By construction $\|\widetilde{\bfu}_n^{\varepsilon}\|_{\mathcal{C}([0,T];L^2(\Om))}\leq K_0$}. Then, using \eqref{u_n_inq} we get 
\bq
\|\widetilde{\bfu}_n^{\varepsilon}\|_{\mathcal{C}([0,T];H^2(\Omega))}\leq K_n \|\widetilde{\bfu}_n^{\varepsilon}\|_{\mathcal{C}([0,T];L^2(\Omega))} \leq K_n K_0.
\eq 

{According to \eqref{estim_density} and the fact that $H^3(\Omega) \xhookrightarrow{} W^{1,\infty}(\Omega)$, we have 
\bq
\begin{aligned}
\max_{t\in [0,T]} \|\nabla\rho_n(t)^{\varepsilon}\|_{L^{\infty}(\Omega)}&\leq K \|\nabla\rho_{0n}\|_{L^{\infty}(\Omega)}\operatorname{exp}\left(\int_0^T K_1 \|\widetilde{\bfu}_n^{\varepsilon}(\tau)\|_{H^3(\Omega)}\md \tau \right)\\
&\leq K \|\nabla\rho_{0n}\|_{L^{\infty}(\Omega)}\operatorname{exp}\left(K_1 K_n K_0 T \right):= K_{\rho}.
\end{aligned}
\eq}

The last inequality comes from the regularity theory of the Stokes operator on smooth domains:

\bq
\| \widetilde{u_n}\|_{C([0,T]; H^3(\Omega))}\leq C(n) \| \widetilde{u_n}\|_{C([0,T]; L^2(\Omega))} \qquad \forall\,\, \widetilde{u_n} \in \bfV_n .
\eq

Next, we have $F_{\varepsilon}(s) \geq 0$ for all $s$ in $\mathbb{R}$, then we control the quantity $\mathcal{R}^{\varepsilon}$ as follows:

\bq
\begin{aligned}
|\mathcal{R}^{\varepsilon}| &\leq \frac{1}{\rho_{*}} \left(\int_{\Omega} \rho_n^{\varepsilon} F_{\varepsilon}(\phi_n^{\varepsilon})\md x \right)\|\widetilde{\bfu}_n^{\varepsilon} - \bfu_n^{\varepsilon}\|_{L^{\infty}(\Omega)}\|\nabla\rho_n^{\varepsilon}\|_{L^{\infty}(\Omega)}\\& \leq \frac{K_n K_{\rho}}{\rho_{*}}E^{\varepsilon}(\rho_n^{\varepsilon},\bfu_n^{\varepsilon},\phi_n^{\varepsilon})\|\widetilde{\bfu}_n^{\varepsilon} - \bfu_n^{\varepsilon}\|_{L^2(\Omega)}\\
& \leq \frac{K_n K_{\rho}}{\rho_{*}}E(\rho_n^{\varepsilon},\bfu_n^{\varepsilon},\phi_n^{\varepsilon})\|\bfu_n^{\varepsilon}\|_{L^2(\Omega)} + \frac{K_n K_{\rho} K_0}{\rho_{*}}E(\rho_n^{\varepsilon},\bfu_n^{\varepsilon},\phi_n^{\varepsilon})\\&
\leq \frac{K_n K_{\rho}}{\rho_{*}}E^{\varepsilon}(\rho_n^{\varepsilon},\bfu_n^{\varepsilon},\phi_n^{\varepsilon})\frac{\|\nabla\bfu_n^{\varepsilon}\|_{L^2(\Omega)}}{\sqrt{\lambda_1^s}} + \frac{K_n K_{\rho} K_0}{\rho_{*}}E^{\varepsilon}(\rho_n^{\varepsilon},\bfu_n^{\varepsilon},\phi_n^{\varepsilon}) \\
& \leq \frac{K_n K_{\rho}}{\rho_{*}}E^{\varepsilon}(\rho_n^{\varepsilon},\bfu_n^{\varepsilon},\phi_n^{\varepsilon})\frac{\sqrt{2}\sqrt{\nu_*}\|\mathbb{D}\bfu_n^{\varepsilon}\|_{L^2(\Omega)}}{\sqrt{\nu_*\lambda_1^s}} + \frac{K_n K_{\rho} K_0}{\rho_{*}}E^{\varepsilon}(\rho_n^{\varepsilon},\bfu_n^{\varepsilon},\phi_n^{\varepsilon})\\&
\leq \int_{\Omega}\frac{\nu(\phi_n^{\varepsilon})}{2}|\mathbb{D}\bfu_n^{\varepsilon}|^2 \md x + \frac{K_n^2 K_{\rho}^2}{\nu_* \rho_*^2 \lambda_1^s} [E^{\varepsilon}(\rho_n^{\varepsilon},\bfu_n^{\varepsilon},\phi_n^{\varepsilon})]^2 + \frac{K_n K_{\rho} K_0}{\rho_{*}}E^{\varepsilon}(\rho_n^{\varepsilon},\bfu_n^{\varepsilon},\phi_n^{\varepsilon})\\
& \leq \int_{\Omega}\frac{\nu(\phi_n^{\varepsilon})}{2}|\mathbb{D}\bfu_n^{\varepsilon}|^2 \md x + \frac{2 K_n^2 K_{\rho}^2}{\nu_* \rho_*^2 \lambda_1^s} [E^{\varepsilon}(\rho_n^{\varepsilon},\bfu_n^{\varepsilon},\phi_n^{\varepsilon} )]^2 + \frac{\nu_* \lambda_1^s K_0^2}{4}
\end{aligned}
\eq
According to \eqref{Energy_Rest}, we arrive at
\bal
\frac{\md }{\md t} E^{\varepsilon}(\rho_n^{\varepsilon},\bfu_n^{\varepsilon},\phi_n^{\varepsilon}) + \int_{\Omega}\frac{\nu(\phi_n^{\varepsilon})}{2}|\mathbb{D}\bfu_n^{\varepsilon}|^2 \md x + \int_{\Omega}D(\widetilde{\phi}_n^{\varepsilon})|\nabla \mu_n^{\varepsilon}|^2 \md x \leq\, &C_1 [E^{\varepsilon}(\rho_n^{\varepsilon},\bfu_n^{\varepsilon},\phi_n^{\varepsilon})]^2\nn\\
&+ C_2K_0^2
\eal
where $ C_1 : = {2 K_n^2 K_{\rho}^2}/{\nu_* \rho_*^2 \lambda_1^s}$ and $ C_2:= {\nu_* \lambda_1^s}/{4}$. In particular, we have 
\bq
\frac{\md }{\md t} E^{\varepsilon}(\rho_n^{\varepsilon},\bfu_n^{\varepsilon},\phi_n^{\varepsilon}) \leq C_1 [E^{\varepsilon}(\rho_n^{\varepsilon},\bfu_n^{\varepsilon},\phi_n^{\varepsilon})]^2 + C_2 K_0^2.
\eq
Integrating over $[0,t]$ for $t<{T}_0$, we get
\bq
E^{\varepsilon}(\rho_n^{\varepsilon},\bfu_n^{\varepsilon},\phi_n^{\varepsilon})(t) \leq E^{\varepsilon}(\rho_{0n},\bfu_{0n},\phi_{0n}) + \int_0^t C_2 K_0^2 \md \tau + C_1\int_0^t [E^{\varepsilon}(\rho_n^{\varepsilon},\bfu_n^{\varepsilon},\phi_n^{\varepsilon})(\tau)]^2\md \tau,
\eq
where 
\bq
E^{\varepsilon}(\rho_{0n},\bfu_{0n},\phi_{0n}) = \int_{\Omega} \frac{1}{2} \rho_{0n} |u_{0n}|^2 + \frac{1}{2}\Ga^2(\nabla\phi_{0n}) + \rho_{0n} F_{\varepsilon}(\phi_{0n}) \md x.
\eq
Now, we apply Bihari’s inequality [see Appendix \ref{A:Bihari's}], with the following choice
\bal
&y(t) = E^{\varepsilon}(\rho_n^{\varepsilon},\bfu_n^{\varepsilon},\phi_n^{\varepsilon})(t),\quad y_0 = E^{\varepsilon}(\rho_{0n},\bfu_{0n},\phi_{0n}) : = E^{\varepsilon}_n(0),
\\ &g(s) = C_2 K_0^2 ,\quad f(s) = C_1 s^2. 
\eal
The anti-derivative of $-\frac{1}{f(s)}$ which cancels at $+\infty$ is $F(s) := \frac{1}{C_1 s}$. Since $F^{-1}(s) = F(s)$, Bihari’s inequality in our case reads

\bq
\sup_{t\leq T^{\prime}} E^{\varepsilon}(\rho_n^{\varepsilon},\bfu_n^{\varepsilon},\phi_n^{\varepsilon})(t)\leq \frac{1}{C_1\left(\frac{1}{C_1(E^{\varepsilon}_n(0) + C_2 K_0^2 T^{\prime}) }- T^{\prime}\right)},\quad \text{for all  } T^{\prime} < \min({T}_0,T^*),
\eq
where $T^* = \frac{1}{C_1(E^{\varepsilon}_n(0) + C_2 T^*)}$. Let $K_0$ be sufficiently large so that $\|\bfu_{0n}\|_{L^2(\Omega)} + \|\phi_{0n}\|_{H^1(\Omega)} \leq K_0^2$. We aim to show that there exists a sufficiently small time $\widetilde{T} = \widetilde{T}(K_0)$ ($\widetilde{T}< \min(T_0,T^*)$) such that
\bq
\|\bfu_{n}\|_{\mathcal{C}([0,\widetilde{T}];L^2(\Omega))} + \|\phi_{n}\|_{\mathcal{C}([0,\widetilde{T}];H^1(\Omega))}\leq K_0^2.
\eq
To this end, we use that
$$\|\bfu^{\varepsilon}_n\|_{L^2(\Omega)}^2 \leq \frac{2}{\rho_*} E^{\varepsilon}(\rho_n^{\varepsilon},\bfu_n^{\varepsilon},\phi_n^{\varepsilon}),\quad \|\nabla\phi_n^{\varepsilon}\|_{L^2(\Omega)}^2 \leq \frac{2}{r} E^{\varepsilon}(\rho_n^{\varepsilon},\bfu_n^{\varepsilon},\phi_n^{\varepsilon}).$$

Since the function $F_{\varepsilon}$ is polynomial of degree two on the domain \\
$\{s\in \mathbb{R}\, ;\, |s|\geq 1-\varepsilon \}$, there exist $c_{\varepsilon}>0$ and $M_{\varepsilon}>0$ such that

\bq
F_{\varepsilon}(s) > M_{\varepsilon} s^2\qquad \text{ for all } \,|s|>c_{\varepsilon},
\eq
which implies that
\bal
\|\phi_n^{\varepsilon}\|^2_{L^2(\Omega)} = \int_{\{|\phi_n^{\varepsilon}|\leq c_{\varepsilon\}}} |\phi_n^{\varepsilon}|^2 \md x + \int_{{ \{|\phi_n^{\varepsilon}| \geq c_{\varepsilon}\}}} |\phi_n^{\varepsilon}|^2 \md x & \leq c_{\varepsilon}^2 |\Om| + \frac{1}{\rho_* M_{\varepsilon}} \int_{\Om} \rho_n^{\varepsilon} F_{\varepsilon}(\phi_n^{\varepsilon})\md x \nn\\
&\leq c_{\varepsilon}^2 |\Om| + \frac{1}{\rho_* M_{\varepsilon}} E^{\varepsilon}(\rho_n^{\varepsilon},\bfu_n^{\varepsilon},\phi_n^{\varepsilon}).
\eal

Thus,
\bq
\begin{aligned}
\|\bfu_{n}^{\varepsilon}\|_{\mathcal{C}([0,\widetilde{T}];L^2(\Omega))} + \|\phi_{n}^{\varepsilon}\|_{\mathcal{C}([0,\widetilde{T}];H^1(\Omega))}&\leq \theta\sup_{t\leq \widetilde{T}} E(\rho_n^{\varepsilon},\bfu_n^{\varepsilon},\phi_n^{\varepsilon})(t) + c_{\varepsilon}^2 |\Om|\\
&\leq \theta \left(\frac{ E_n(0) +  C_2 K_0^2 \widetilde{T}}{1 - C_1 E_n(0)\widetilde{T} - C_1 C_2 K_0^2 \widetilde{T}^2}\right)+c_{\varepsilon}^2 |\Om|
\end{aligned}
\eq
where $\theta:= \max\left\{\frac{2}{\rho_*},\frac{1}{\rho_* M_{\varepsilon}},\frac{2}{r}\right\}$. The following condition should be satisfied
\bq
\theta \left(\frac{ E_n^{\varepsilon}(0) + C_2 K_0^2 \widetilde{T}}{1 - C_1 E_n^{\varepsilon}(0)\widetilde{T} - C_1 C_2 K_0^2 \widetilde{T}^2}\right)+c_{\varepsilon}^2 |\Om| \leq K_0^2,
\eq
which can be written as
\bq
E_n^{\varepsilon}(0) + C_2 K_0^2\widetilde{T} \leq \frac{1}{\theta} (K_0^2 - c_{\varepsilon}^2 |\Om|) - C_1 E_n^{\varepsilon}(0) \frac{1}{\theta}(K_0^2 - c_{\varepsilon}^2 |\Om|) \widetilde{T} - C_1 C_2 K_0^2 \frac{1}{\theta} (K_0^2 - c_{\varepsilon}^2 |\Om|)\widetilde{T}^2.
\eq
Since we can take $\widetilde{T} \leq 1$ the condition holds if
\bal
E_n^{\varepsilon}(0) + \left(C_2 K_0^2 + C_1 E_n^{\varepsilon}(0) \frac{1}{\theta}(K_0^2 -c_{\varepsilon}^2 |\Om|) + C_1 C_2 \frac{1}{\theta} K_0^2(K_0^2 - c_{\varepsilon}^2 |\Om|)\right)\widetilde{T}  \nn\\
\hspace{5cm}\leq\frac{1}{\theta} (K_0^2 - c_{\varepsilon}^2 |\Om|).\label{condition}
\eal
Thus, by taking $E_n^{\varepsilon}(0) < \frac{1}{\theta}(K_0^2 - c_{\varepsilon}^2 |\Omega|)$, there exists $\widetilde{T} = \widetilde{T}(K_0)$ sufficiently small such that inequality \eqref{condition} holds. We conclude that
\bq
\|\bfu_{n}^{\varepsilon}\|_{\mathcal{C}([0,\widetilde{T}];L^2(\Omega))} + \|\phi_{n}^{\varepsilon}\|_{\mathcal{C}([0,\widetilde{T}];H^1(\Omega))}\leq K_0^2.
\eq

For any set of parameters $(n,{\varepsilon},\rho_*,\rho^*, E_n^{\varepsilon}(0), \nu_*,\rho_{0n})$, let now introduce the set
$$ \mathcal{S}_n = \{ (\bfu_n^{\varepsilon} ,\phi_n^{\varepsilon}) \in \mathcal{C}([0,\widetilde{T}];\bfV_n\times H_n) : \|\bfu_n^{\varepsilon}\|_{\mathcal{C}([0,\widetilde{T}];L^2(\Omega))} + \|\phi_n^{\varepsilon}\|_{\mathcal{C}([0,\widetilde{T}];H^1(\Omega))}\leq K_0^2\}.$$

\textbf{Step 3:} Next, we show that the derivatives of $\bfu_n^{\varepsilon}$ and $\phi_n^{\varepsilon}$ with respect to time are also bounded on $[0,\widetilde{T}]$. Multiplying \eqref{velocity} by $\frac{\md}{\md t}\alpha_i^{\varepsilon}(t)$ and summing over $i$, we find
$$
\begin{aligned}
\rho_*  \|\partial_t \bfu_n^{\varepsilon}\|_{L^2(\Omega)}^2\leq  & C \rho^*\|\widetilde{\bfu}_n^{\varepsilon}\|_{H^2(\Omega)}\|\nabla \bfu_n^{\varepsilon}\|_{L^2(\Omega)}\|\partial_t \bfu_n^{\varepsilon}\|_{L^2(\Omega)}\\
&+\nu^*\|\nabla \bfu_n^{\varepsilon}\|_{L^2(\Omega)}\|\nabla \partial_t \bfu_n^{\varepsilon}\|_{L^2(\Omega)}\\
& +C \rho^*\|\mu_n^{\varepsilon}\|_{L^2(\Omega)}\|\widetilde{\phi}_n^{\varepsilon}\|_{H^2(\Omega)}\|\nabla \partial_t \bfu_n^{\varepsilon}\|_{L^2(\Omega)}\\&+\rho^* C\left(1+\|\phi_n^{\varepsilon}\|_{H^2(\Omega)}^4\right)\|\partial_t \bfu_n^{\varepsilon}\|_{L^2(\Omega)}.
\end{aligned}.
$$
We multiply \eqref{chemical} by $\gamma_i^{\varepsilon}(t)$ and sum over $i$, We find
$$\begin{aligned}
\rho_{*} \|\mu_n^{\varepsilon}\|^2_{L^2(\Omega)} &\leq \left(\rho^* \|F_{\varepsilon}^{\prime} (\phi_n^{\varepsilon})\|_{L^2(\Omega)}+\|\operatorname{div}(\Ga\bfxi(\nabla\phi_n^{\varepsilon})\|_{L^2(\Omega)}\right)\|\mu_n^{\varepsilon}\|_{L^2(\Omega)}\\
&\leq \left(C\rho^* \|\phi_n^{\varepsilon}\|^3_{H^1(\Omega)} +C \|\phi_n^{\varepsilon}\|_{H^2(\Omega)}\right)\|\mu_n^{\varepsilon}\|_{L^2(\Omega)}\\
&\leq \left(C\rho^* \|\phi_n^{\varepsilon}\|^3_{H^1(\Omega)}+C \lambda_n\|\phi_n^{\varepsilon}\|_{H^1(\Omega)}\right)\|\mu_n^{\varepsilon}\|_{L^2(\Omega)}.
\end{aligned}
$$
Thus,
\bq
\|\mu_n\|_{L^2(\Omega)} \leq C_{\mu}\label{estim_nu},
\eq
where $C_{\mu} = C_{\mu}(n,{\varepsilon},\rho_*,\rho^*, \lambda_n, K_0) $. We deduce that there exist a constant $K_1 = K_1(n,{\varepsilon},\rho_*,\rho^*, \nu^*,K_0)$ such that
\bq
\max_{[0,\widetilde{T}]}\|\partial_t\bfu_n^{\varepsilon}(t)\|_{L^2(\Omega)}\leq K_1^2.
\eq
Similarly, multiplying \eqref{concentration} by $\frac{\mathrm{d}}{\mathrm{d} t} \beta_i^{\varepsilon}$ and summing over $i$, we get
$$
\rho_*\|\partial_t \phi_n^{\varepsilon}\|^2_{L^2(\Omega)} 
\leq\left(C \rho^* \sqrt{\lambda_n^s} \sqrt{\lambda_n}\|\bfu_n^{\varepsilon}\|_{L^2(\Omega)}\|\widetilde{\phi}_n^{\varepsilon}\|_{H^1(\Omega)}+D_{*}\lambda_n\|\mu_n^{\varepsilon}\|_{L^2(\Omega)}\right)\|\partial_t \phi_n^{\varepsilon}\|_{L^2(\Omega)} .
$$

Then, there exists $K_2=K_2\left(n,{\varepsilon}, \rho_*, \rho^*, K_0\right)$ such that
$$
\max _{[0,\widetilde{T}]}\|\partial_t \phi_n^{\varepsilon}(t)\|_{L^2(\Omega)} \leq K_2^2 .
$$

Let us set $\widetilde{K}_0^2=K_1^2+K_2^2$. We consider the subset $\widetilde{\mathcal{S}}_n$ of $\mathcal{S}_n$ defined as
$$
\begin{aligned}
&\widetilde{\mathcal{S}}_n=\left\{(\bfv_n^{\varepsilon}, \psi_n^{\varepsilon}) \in \mathcal{C}^1\left([0, \widetilde{T}] ; \mathbf{V}_n \times H_n\right):\right.  \max _{[0,\widetilde{T}]}\|\bfv_n^{\varepsilon}(t)\|_{L^2(\Omega)}+\max _{[0,\widetilde{T}]}\|\psi_n^{\varepsilon}(t)\|_{H^1(\Omega)} \leq K_0^2,\\
&\hspace{5.3cm}\left.\max_{[0,\widetilde{T}]}\|\partial_t \bfv_n^{\varepsilon}(t)\|_{L^2(\Omega)}+\max _{[0,\widetilde{T}]}\|\partial_t \psi_n^{\varepsilon}(t)\|_{H^1(\Omega)} \leq \widetilde{K}_0^2\right\}
\end{aligned}
$$
By the Ascoli-Arzelà theorem, the set $\widetilde{\mathcal{S}}_n$ is compact in $\mathcal{S}_n$. We consider the map $\Lambda$ defined as
$$
\Lambda : {\mathcal{S}}_n\longmapsto \widetilde{\mathcal{S}}_n,\qquad (\widetilde{\bfu}_n^{\varepsilon} , \widetilde{\phi}_n^{\varepsilon}) \longmapsto
\Lambda(\widetilde{\bfu}_n^{\varepsilon} , \widetilde{\phi}_n^{\varepsilon}) := (\bfu_n^{\varepsilon},\phi_n^{\varepsilon}). $$
Since the map $\bfp\longmapsto\Ga\bfxi$ is linear and continuous, $\Lambda$ is a continuous map on ${\mathcal{S}}_n$. This result can be proved following the same method as in \cite[see Appendix A]{giorgini2020weak}. Now we show that the solution is defined on the whole interval $[0,T]$ for any fixed $T>0$, instead of $\widetilde{T}$.
\end{proof}
\begin{proposition}\label{prop2}
Given the initial data $(\rho_{0n},\bfu_{0n},\phi_{0n},\mu_{0n})$ and let $(\rho_{n}^{\varepsilon},\bfu_{n}^{\varepsilon},\phi_{n}^{\varepsilon},\mu_{n}^{\varepsilon})$ denote a solution of \eqref{density_T}-\eqref{chemical_T} defined in $[0,T_0)$. There is a constant
$K_0$, depending on $(\rho_{*},\rho^{*},\|\bfu_{0}\|_{L^2(\Omega)},\|\phi_{0}\|_{H^1(\Omega)})$, but independent of $T_0$ , such that 
\bq
\|\bfu_{n}^{\varepsilon}\|_{\mathcal{C}([0,{T}_0];L^2(\Omega))} + \|\phi_{n}^{\varepsilon}\|_{\mathcal{C}([0,{T}_0];H^1(\Omega))}\leq K_0^2.
\eq
\end{proposition}
\begin{proof}\label{proof_prop_2}
Multiplying \eqref{velocity_T} by $\alpha_i^{\varepsilon}$ and summing over $i$ we find
\bq
\begin{aligned}
\frac{ \md}{ \md t}\int_{\Omega} \frac{1}{2}\rho_n^{\varepsilon} |\bfu_n^{\varepsilon}|^2 \md x + \int_{\Omega}\nu({\phi}_n^{\varepsilon})|\mathbb{D}(\bfu_n^{\varepsilon})|^2 \md x=\int_{\Omega}& \rho_n^{\varepsilon} \mu_n^{\varepsilon}\bfu_n^{\varepsilon}\cdot\nabla{\phi}_n^{\varepsilon}\md x\\
&-\int_{\Omega} \rho_n^{\varepsilon} \bfu_n^{\varepsilon}\cdot\nabla F_{\varepsilon}(\phi_n^{\varepsilon})\md x.\label{KIN3}
\end{aligned}
\eq
Now, we multiply
\eqref{concentration_T} by $\gamma_i^{\varepsilon}$ and \eqref{chemical_T} by $\frac{\md }{\md t} \beta_i^{\varepsilon}$, respectively, and summing over $i$, we get
\bq 
\begin{aligned}
\frac{\md }{\md t} \int_{\Omega} \left[\frac{1}{2}\Ga^2(\nabla\phi_n^{\varepsilon}) + \rho_n^{\varepsilon} F_{\varepsilon}(\phi_n^{\varepsilon})\right] \md x &+ \int_{\Omega}\rho_n^{\varepsilon} \bfu_n^{\varepsilon}\cdot\nabla{\phi}_n^{\varepsilon}\mu_n^{\varepsilon}\md x &+ \int_{\Omega}D(\phi_n^{\varepsilon})|\nabla\mu_n^{\varepsilon}|^2 \md x\\
&= \int_{\Omega} \partial_t \rho_n^{\varepsilon} F_{\varepsilon}(\phi_n^{\varepsilon})\md x. \label{Pot_1}
\end{aligned}
\eq
By summing the previous equation \eqref{KIN3} and \eqref{Pot_1}, we obtain 

\bal
\frac{\md}{\md t} \int_{\Omega} \frac{1}{2} \rho_n^{\varepsilon}\left|\bfu_n^{\varepsilon}\right|^2+\frac{1}{2}\Ga^2\left(\nabla \phi_n^{\varepsilon}\right)&+\rho_n^{\varepsilon} F_{\varepsilon}\left(\phi_n^{\varepsilon}\right) \md x\nn\\
&+\int_{\Omega} \nu\left(\phi_n^{\varepsilon}\right)\left|\mathbb{D} \bfu_n^{\varepsilon}\right|^2+D(\phi_n^{\varepsilon})\left|\nabla \mu_n^{\varepsilon}\right|^2 \md x=0 .
\eal

Integrating in time over $[0, t]$ where $0<t<T_0$, we deduce that
\bq
\begin{aligned}
\int_{\Omega}& \frac{1}{2} \rho_n^{\varepsilon}(t)\left|\bfu_n^{\varepsilon}(t)\right|^2 +\frac{1}{2}\Ga^2\left(\nabla \phi_n^{\varepsilon}(t)\right)+\rho_n^{\varepsilon}(t) F_{\varepsilon}\left(\phi_n^{\varepsilon}(t)\right) \md x\\
&+\int_0^t \int_{\Omega} \nu\left(\phi_n^{\varepsilon}(\tau)\right)\left|\mathbb{D} \bfu_n^{\varepsilon}(\tau)\right|^2 \md x \md \tau +\int_0^t \int_{\Omega}D(\phi_n^{\varepsilon}(\tau))\left|\nabla \mu_n^{\varepsilon}(\tau)\right|^2 \md x \md \tau\\
&=\int_{\Omega} \frac{1}{2} \rho_{0 n}\left|\bfu_{0 n}\right|^2+\frac{1}{2}\Ga^2\left(\nabla \phi_{0 n}\right)+\rho_{0 n} F_{\varepsilon}\left(\phi_{0 n}\right) \mathrm{d} x.\label{Enrg_n}
\end{aligned}
\eq

Using the properties of the projector operator $P_n^1$ and $P_n^2$, we find
$$
\begin{aligned}
\int_{\Omega} \frac{1}{2} &\rho_{0 n}\left|\bfu_{0 n}\right|^2  +\frac{1}{2}\Ga^2\left(\nabla \phi_{0 n}\right)+\rho_{0 n} F_{\varepsilon}\left(\phi_{0 n}\right) \mathrm{d} x \\& \leq \frac{\rho^*}{2}\|\bfu_0\|_{L^2(\Omega)}^2+\frac{R}{2}\|\phi_0\|_{H^1(\Omega)}^2+K({\varepsilon}) \rho^*\left(1+\|\phi_0\|_{H^1(\Omega)}^4\right),
\end{aligned}
$$
where the constant $K({\varepsilon}) $ is independent of $n$. Since $F_{\varepsilon}(s) \geq 0$, we obtain
$$
\begin{aligned}
\frac{\rho_*}{2}\|\bfu_n^{\varepsilon}(t)\|_{L^2(\Omega)}^2 +\frac{r}{2}\|\nabla \phi_n^{\varepsilon}(t)\|_{L^2(\Omega)}^2 \leq \frac{\rho^*}{2}\|\bfu_0\|_{L^2(\Omega)}^2&+\frac{R}{2}\|\phi_0\|_{H^1(\Omega)}^2\\
&+K({\varepsilon}) \rho^*\left(1+\|\phi_0\|_{H^1(\Omega)}^4\right).
\end{aligned}
$$

Therefore, there exists a positive constant $K_0=K_0\left({\varepsilon},\rho_*, \rho^*, \bfu_0, \phi_0,r,R\right)$ such that
$$
\sup _{[0 ,T_0]}\|\bfu_n^{\varepsilon}(t)\|_{L^2(\Omega)}+\sup_{[0 ,T_0]}\|\phi_n^{\varepsilon}(t)\|_{H^1(\Omega)} \leq K_0^2.
$$
\end{proof}
The proposition \ref{prop2} shows that if a solution exists on the time interval $[0,T_0)$, it cannot blow up as $T_0$. Consequently, we conclude that the Galerkin approximate sequence exists globally that solves the system \eqref{density_T}-\eqref{chemical_T}.

\section{Proof of theorem \ref{th1}}\label{Global existence}
In this section we prove Theorem \ref{th1}: the proof is split into three parts. Frist, we derive
estimates on the solutions of the Galerkin approximate problem \eqref{density_T}-\eqref{chemical_T}, independently of $n$. SEcond, we pass to the limit as $n\rightarrow \infty$ and recover weak solutions of the regularized problem. Third, following the sqme approach, we derive estimates on solutions of the regularized problem \eqref{rho_epsilon}-\eqref{muepsilon} independently of $\varepsilon$. We
pass to the limit as $\varepsilon\rightarrow 0^{+}$ and recover weak solutions of the problem \eqref{rho}-\eqref{mu}.

Starting from \eqref{Enrg_n} we obtain
\bq
\begin{aligned}
E^{\varepsilon}(\rho_n^{\varepsilon}(t),\bfu_n^{\varepsilon}(t),\phi_n^{\varepsilon}(t)) +\int_0^t \int_{\Omega} \nu\left(\phi_n^{\varepsilon}(\tau)\right)\left|\mathbb{D} \bfu_n^{\varepsilon}(\tau)\right|^2 +&D(\phi_n^{\varepsilon}(\tau))\left|\nabla \mu_n^{\varepsilon}(\tau)\right|^2 \md x \md \tau \\
&= E^{\varepsilon}(\rho_{0n},\bfu_{0n},\phi_{0n}).
\end{aligned}
\eq
From \eqref{phi_0m} it follows that $E^{\varepsilon}(\rho_{0n},\bfu_{0n},\phi_{0n})$ converges to $E^{\varepsilon}(\rho_{0},\bfu_{0},\phi_{0})$ as $n\rightarrow\infty$.
where
$$
E^{\varepsilon}(\rho_{0},\bfu_{0},\phi_{0}) = \int_{\Omega} \frac{1}{2} \rho_{0} |u_{0}|^2 + \frac{1}{2}\Ga^2(\nabla\phi_{0}) + \rho_{0} F_{\varepsilon}(\phi_{0}) \md x := E^{\varepsilon}_0.
$$
Since $F_{\varepsilon} \geq 0$, we have for $n$ sufficiently large
\bq
\|\bfu_n^{\varepsilon}(t)\|^2_{L^2(\Omega)} + \|\nabla\phi_n^{\varepsilon}(t)\|^2_{L^2(\Omega)} \leq \max\left\{\frac{2}{r},\frac{2}{\rho_*}\right\} E^{\varepsilon}(\rho_0,\bfu_0,\phi_0).
\eq

Consequently, there exists a constant $C = C(\varepsilon,\rho_*, \nu_*,D_*,r)$ such that

\bal
&\|\bfu_n^{\varepsilon}\|_{L^{\infty}(0,T;\bfH)}\leq C\sqrt{E_0^{\varepsilon}},\quad \|\bfu_n^{\varepsilon}\|_{L^2(0,T;\bfV)}\leq C\sqrt{E_0^{\varepsilon}},\label{u_Linfity_H_V}\quad\\
&\|\nabla\phi_n^{\varepsilon}\|_{L^{\infty}(0,T;L^2(\Omega))}\leq C\sqrt{E_0^{\varepsilon}},
\quad\|\nabla\mu_n^{\varepsilon}\|_{L^2(0,T;L^2(\Omega))}\leq C\sqrt{E_0^{\varepsilon}}.\label{Grad_estimate}
\eal

Similar estimates to those in the isotropic case can be obtained \cite{giorgini2020weak}

\bal
&\|\rho_n^{\varepsilon}\|_{L^{\infty}(0,T;H^{-1}(\Omega))} \leq C,\label{estim11}\quad\|\phi_n^{\varepsilon}\|_{L^{\infty}(0,T;H^1(\Omega))} \leq C ,\quad\|\mu_n^{\varepsilon}\|_{L^{4}(0,T;L^2(\Omega))} \leq C,\quad \\
&\|\mu_n^{\varepsilon}\|_{L^{2}(0,T;H^1(\Omega))} \leq C,\quad
\|\phi_n^{\varepsilon}\|_{B_{\infty,\infty}^{\frac{1}{4}}(0,T;L^2(\Omega))} \leq C.\label{phi_B_infity_infity}
\eal

For some positive constant $C:= C(\varepsilon,T)$ independent of $n$.

{  

Since $\omega_1 = 1$, we infer from \eqref{density_T} and \eqref{concentration_T} that

\[
\left| \int_{\Omega} \rho_n^{\varepsilon}(t)\,\phi_n^{\varepsilon}(t)\,\md x \right|
=
\left| \int_{\Omega} \rho_{0n}^{\varepsilon} \phi_{0n}^{\varepsilon} \, \md x \right|
\le
C(\varepsilon),\quad \forall t \in [0,T].
\]

In light of the generalized Poincaré inequality (see \cite[Chapter II, Section 1.4]{temam2012infinite}), together with \eqref{Grad_estimate}, the above inequality implies that

\bq
\|\phi_n^{\varepsilon}\|_{L^{\infty}(0,T;H^1({\Om}))} \le C(\varepsilon,T). \label{phi_time_infity}
\eq

}

To estimate $\|\bfu_n^{\varepsilon}\|_{B_{2,\infty}^{\frac{1}{4}}(0,T;\bfH)}$, let $h$ such that $0 <h< T$. For $0 \leq t \leq T - h$ we test \eqref{velocity_T} at time $\tau$ by $\bfu_n^{\varepsilon} (t+h) - \bfu_n^{\varepsilon} (t)$ and \eqref{density_T} at time $\tau$ by $\bfu_n^{\varepsilon}(t)\cdot(\bfu_n^{\varepsilon}(t+h) - \bfu_n^{\varepsilon} (t))$. After integration with
respect to $\tau$ from $t$ to $t+h$ and some obvious manipulations, we arrive at

\bq
\begin{aligned}\label{time_translation}
 \int_{\Omega} \rho_n^{\varepsilon}(t+h)&\left|\bfu_n^{\varepsilon}(t+h)-\bfu_n^{\varepsilon}(t)\right|^2 \md x
 =\\ 
&\underbrace{\int_{\Omega}-\left(\rho_n^{\varepsilon}(t+h)-\rho_n^{\varepsilon}(t)\right) \bfu_n^{\varepsilon}(t) \cdot\left(\bfu_n^{\varepsilon}(t+h)-\bfu_n^{\varepsilon}(t)\right) \mathrm{d} x}_{I_1(t)} \\
& +\underbrace{\int_t^{t+h} \int_{\Omega}-\operatorname{div}\left(\rho_n^{\varepsilon}(\tau) \bfu_n^{\varepsilon}(\tau) \otimes \bfu_n^{\varepsilon}(\tau)\right) \cdot\left(\bfu_n^{\varepsilon}(t+h)-\bfu_n^{\varepsilon}(t)\right) \mathrm{d} x \md \tau}_{I_2(t)} \\
& +\underbrace{\int_t^{t+h} \int_{\Omega}-\nu\left(\phi_n^{\varepsilon}(\tau)\right) \mathbb{D} \bfu_n^{\varepsilon}(\tau): \nabla\left(\bfu_n^{\varepsilon}(t+h)-\bfu_n^{\varepsilon}(t)\right) \mathrm{d} x \md \tau}_{I_3(t)} \\
& +\underbrace{\int_t^{t+h} \int_{\Omega} \rho_n^{\varepsilon}(\tau) \mu_n^{\varepsilon}(\tau) \nabla \phi_n^{\varepsilon}(\tau) \cdot\left(\bfu_n^{\varepsilon}(t+h)-\bfu_n^{\varepsilon}(t)\right) \mathrm{d} x \md \tau}_{I_4(t)} \\
& +\underbrace{\int_t^{t+h} \int_{\Omega}-\rho_n^{\varepsilon}(\tau) F_{\varepsilon}^{\prime}\left(\phi_n^{\varepsilon}(\tau)\right) \nabla \phi_n^{\varepsilon}(\tau) \cdot\left(\bfu_n^{\varepsilon}(t+h)-\bfu_n^{\varepsilon}(t)\right) \mathrm{d} x \md \tau}_{I_5(t)}. 
\end{aligned}
\eq

In the isotropic case ($\Gamma^2(\nabla\phi) = |\nabla\phi|$), we can control $\|\phi_n^{}\|_{H^2(\Omega)}$ independently of $n$. However, in the anisotropic case, this argument is not applicable. Alternatively, to estimate $\|\bfu_n^{}\|_{B_{2,\infty}^{\frac{1}{4}}(0,T;\bfH)}$,
it follows along the same lines of \cite[page 210]{giorgini2020weak}  that 
$\int_{0}^{T-h} |I_k(t)| \md t \leq C h^{\frac{1}{2}}$ for $k=1,2,3\text{ and } 4$, with $C$ is a constant independent of $n$. Thus, it is sufficient to control $I_5(t)$
$$
I_5(t) = \int_t^{t+h} \int_{\Omega}-\rho_n^{\varepsilon}(\tau) F_{\varepsilon}^{\prime}\left(\phi_n(\tau)\right) \nabla \phi_n^{\varepsilon}(\tau) \cdot\left(\bfu^{\varepsilon}_n(t+h)-\bfu^{\varepsilon}_n(t)\right) \md x \md \tau.$$

$$
\begin{aligned}
\left|I_5(t)\right| & \leq \rho^* \int_t^{t+h}\left\|F_{\varepsilon}^{\prime}\left(\phi_n^{\varepsilon}(\tau)\right)\right\|_{L^6(\Omega)}\left\|\nabla \phi_n^{\varepsilon}(\tau)\right\|_{L^2(\Omega)} \md \tau\left\|\bfu^{\varepsilon}_n(t+h)-\bfu^{\varepsilon}_n(t)\right\|_{L^3(\Omega)} \\
& \leq C\int_t^{t+h}\left\|F_{\varepsilon}^{\prime}\left(\phi_n^{\varepsilon}(\tau)\right)\right\|_{L^6(\Omega)} \mathrm{d} \tau\left(\left\|\nabla \bfu^{\varepsilon}_n(t+h)\right\|_{L^2(\Omega)}^{\frac{1}{2}}+\left\|\nabla \bfu^{\varepsilon}_n(t)\right\|_{L^2(\Omega)}^{\frac{1}{2}}\right) .
\end{aligned}
$$
Integrating from $0$ to $T-h$, we find
$$
\int_0^{T-h}\left|I_5(t)\right| \mathrm{d} t \leq C h^{\frac{3}{4}} \int_0^T\left\|F_{\varepsilon}^{\prime}\left(\phi_n^{\varepsilon}(\tau)\right)\right\|_{L^6(\Omega)} \mathrm{d} \tau.
$$
We have 
\begin{align*}
&\left\|F_{\varepsilon}^{\prime}\left(\phi_n^{\varepsilon}(\tau)\right)\right\|^6_{L^6(\Omega)} = \int_{|\phi_n^{\varepsilon}|\leq 1-\varepsilon}|F_{\varepsilon}^{\prime}\left(\phi_n^{\varepsilon}(\tau)\right)|^6 \md x + \int_{|\phi_n^{\varepsilon}|\geq 1-\varepsilon}|F_{\varepsilon}^{\prime}\left(\phi_n^{\varepsilon}(\tau)\right)|^6 \md x\\
&\hspace{1cm}\leq \left({\lambda_2}(1-\varepsilon) + \frac{\lambda_1}{2}\log(1+\varepsilon/2) + |\log(\varepsilon/2)|\right)^6 |\Om|  \\
&\hspace{3cm}+\int_{|\phi_n^{\varepsilon}|\geq 1-\varepsilon}|F_{\varepsilon}^{\prime}\left(\phi_n^{\varepsilon}(\tau)\right)|^6 \md x
\end{align*}

On the domain $\{s\in \mathbb{R}\, ;\, |s|\geq 1-\varepsilon \}$, the function $F_{\varepsilon}^{\prime}$ is polynomial of degree one. {  According to \eqref{phi_time_infity} 
and the fact that $H^1(\Om)\xhookrightarrow{} L^6(\Om)$ (for $d=2,3$)}
, we get
$$
\int_0^{T-h}\left|I_5(t)\right| \mathrm{d} t \leq C(\varepsilon)h^{\frac{3}{4}} .
$$
Then, we deduce that
\bq
\|\bfu_n^{\varepsilon}\|_{B_{2,\infty}^{\frac{1}{4}}(0,T;\bfH)} \leq C(\varepsilon,T).\label{u_B_14}
\eq

With the uniform estimates \eqref{estim11}-\eqref{u_B_14} in hand,  a compactness argument shows that a subsequence of the approximate sequence converges to some function
$(\rho^{\varepsilon},\bfu^{\varepsilon},\phi^{\varepsilon},\mu^{\varepsilon})$ satisfying system \eqref{rho_epsilon}-\eqref{muepsilon} in the sense of distributions. Since the proof closely follows \cite{giorgini2020weak}, we omit the details for brevity. {The key difference from \cite{giorgini2020weak} lies in passing to the limit of the quantity \( \Gamma \bfxi(\nabla \phi_n^{\varepsilon}) \otimes \nabla \phi_n^{\varepsilon} \). 
{  We first show that \( \phi_n^{\varepsilon} \rightarrow \phi^{\varepsilon} \) strongly in \( L^2(0,T; H^1(\Omega)) \). To this end, we will use an argument presented in \cite{Garcke}.

For any given scalar-valued or vector-valued function $f$ defined on $\Omega$, we write
$\partial^{+h}_i f$ and $\partial^{-h}_i f$ to denote the difference quotients at a
given point $x \in \Omega$ by
\[
\partial^{+h}_i f(x)
:= \frac{f(x + h e_i) - f(x)}{h},
\qquad
\partial^{-h}_i f(x)
:= \frac{f(x) - f(x - h e_i)}{h}.
\]
According to the work \cite[see page 12]{Garcke}, if assumption ${\bfH_3}$ holds, one can prove that 
\[
a_\Ga \int_{\Omega} \lvert D^2 \phi_n^{\varepsilon} \rvert^2\md x
\le
\int_{\Omega} \operatorname{div}\!\bigl(\Ga\bfxi(\nabla \phi_n^{\varepsilon})\bigr)\,\Delta \phi_n^{\varepsilon}\md x.
\]

We consider in \eqref{chemical}, $\om_i = \De\phi_n^{\varepsilon}$ we get 
\[
\int_\Om\rho_n^{\varepsilon}\mu_n^{\varepsilon}\De\phi_n^{\varepsilon} \md x= \int_\Om\Ga\bfxi(\nabla\phi_n^{\varepsilon})\nabla\De\phi_n^{\varepsilon}\md x + \int_\Om\rho_n^{\varepsilon} F_{\varepsilon}^{\prime}(\phi_n^{\varepsilon})\De\phi_n^{\varepsilon}\md x,
\]
by integration by part, we get
\[
\int_\Om\operatorname{div}\left(\Ga\bfxi(\nabla\phi_n^{\varepsilon})\right)\De\phi_n^{\varepsilon} \md x= -\int_\Om\rho_n^{\varepsilon}\mu_n^{\varepsilon}\De\phi_n^{\varepsilon}\md x  + \int_\Om\rho_n^{\varepsilon} F_{\varepsilon}^{\prime}(\phi_n^{\varepsilon})\De\phi_n^{\varepsilon}\md x,
\]
hence,
\bal
a_\Ga &\int_{\Omega} \lvert D^2 \phi_n^{\varepsilon} \rvert^2 \md x
\le  \rho^* \|\mu_n^{\varepsilon}\|_{L^2(\Om)}\|\De\phi_n^{\varepsilon}\|_{L^2(\Om)}+\rho^*\left(\int_\Om |F_{\varepsilon}^{\prime}(\phi_n^{\varepsilon})|^2\,\md x\right)^{1/2}\|\De \phi_n^{\varepsilon}\|_{L^2(\Om)}\nn\\
&\le \rho^* \|\mu_n^{\varepsilon}\|_{L^2(\Om)}\|\De\phi_n^{\varepsilon}\|_{L^2(\Om)} + \rho^* \left(\left({\lambda_2}(1-\varepsilon) + \frac{\lambda_1}{2}\log(1+\varepsilon/2) + |\log(\varepsilon/2)|\right)^2 |\Om|\right.\nn\\
&\hspace{5cm}+ \left.\int_{|\phi_n^{\varepsilon}|\geq 1-\varepsilon}|F_{\varepsilon}^{\prime}\left(\phi_n^{\varepsilon}\right)|^2 \md x \right)^{1/2} \|\De\phi_n^{\varepsilon}\|_{L^2(\Om)}.
\eal

Same here, On the domain $\{s\in \mathbb{R}\, ;\, |s|\geq 1-\varepsilon \}$, the function $F_{\varepsilon}^{\prime}$ is polynomial of degree one. 
In light of \eqref{phi_time_infity} and the fact that $H^1(\Om)\xhookrightarrow{} L^2(\Om)$ (for $d=2,3$)
, we obtain

\bq
\|\De \phi_n^{\varepsilon}\|_{L^2(\Om)}\le \rho^* (\|\mu_n^{\varepsilon}\|_{L^2(\Om)} + C(\varepsilon)). \label{Delta_phi_L2}
\eq
Now, by taking $\om = \mu_n^{\varepsilon}$ in \eqref{chemical}, we obtain
\bal 
\rho_* \, \lVert \mu_n^{\varepsilon} \rVert_{L^2(\Omega)}^2
&\le
\lVert \Ga\bfxi(\nabla \phi_n^{\varepsilon}) \rVert_{L^2(\Omega)}
\, \lVert \nabla \mu_n^{\varepsilon}\rVert_{L^2(\Omega)}
+ \rho^* \, \lVert F_{\varepsilon}'(\phi_n^{\varepsilon}) \rVert_{L^2(\Omega)}
\, \lVert \mu_n^{\varepsilon}\rVert_{L^2(\Omega)}\\
&\le C \lVert \nabla \phi_n^{\varepsilon}\rVert_{L^2(\Omega)}
\, \lVert \nabla \mu_n^{\varepsilon}\rVert_{L^2(\Omega)}
+ \rho^* C(\varepsilon)\, \lVert \mu_n^{\varepsilon}\rVert_{L^2(\Omega)}
\eal
In light of \eqref{phi_time_infity}, we have 
\[
\|\mu_n^{\varepsilon}\|^2_{L^2(\Om)}\le C(\varepsilon)(1+ \|\na\mu_n^{\varepsilon}\|_{L^2(\Om)}).
\]
Thus, from \eqref{Grad_estimate} and \eqref{phi_time_infity} we deduce that
\bq
\|\mu_n^{\varepsilon}\|_{L^4(0,T;L^2(\Om))} \le C(\varepsilon,T), \qqquad \|\mu_n^{\varepsilon}\|_{L^2(0,T;H^1(\Om))} \le C(\varepsilon,T).\label{muL4}
\eq
Now taking \eqref{Delta_phi_L2} together with \eqref{muL4} the above estimate yields 

\bq
\|\phi_n^{\varepsilon}\|_{L^2(0,T;H^2(\Om))} \le C(\varepsilon,T).\label{phi_L4_H2}
\eq

Now, According to \cite[Lemma 3.1]{giorgini2020weak}, together with \eqref{phi_B_infity_infity}-- \eqref{phi_L4_H2} and \eqref{u_Linfity_H_V}--\eqref{u_B_14} we have 

\begin{align}
\bfu_n^{\varepsilon} &\to \bfu^{\varepsilon} 
&& \text{strongly in } L^2\bigl(0,T; L^q(\Omega)\bigr),
\quad \forall\, q \in [2,6),  \\[0.5em]
\phi_n^{\varepsilon} &\to \phi^{\varepsilon} 
&& \text{strongly in } L^2\bigl(0,T; W^{1,q}(\Omega)\bigr),
\quad \forall\, q \in [2,6), \\[0.5em]
\phi_n^{\varepsilon} &\to \phi^{\varepsilon} 
&& \text{strongly in } C\bigl([0,T]; L^q(\Omega)\bigr),
\quad \forall\, q \in [2,6).
\end{align}

}

{  
According to assumption \((\bfH_2)\), the map \( \mathbf{p} \mapsto \Gamma \bfxi(\mathbf{p}) \) is linear, and thus continuous, which implies that 

\bq
\Gamma \bfxi(\nabla \phi_n^{\varepsilon}) \otimes \nabla \phi_n^{\varepsilon} \rightarrow \Gamma \bfxi(\nabla \phi^{\varepsilon}) \otimes \nabla \phi^{\varepsilon}  \text{ strongly in  }  L^1(0,T; L^2(\Omega)).\label{pass_limit}
\eq 
}

{ 

Furthermore, the equation \eqref{muepsilon} holds almost everywhere in $(0,T)\times\Om$.}

\begin{proposition}
\label{prop12}
There exists a constant $C(T) > 0$ such that, for all $\varepsilon\in (0,1 - \sqrt{1-{\lambda_2}/{\lambda_1}})$, it holds
\bq
\|\rho^{\varepsilon}\|_{L^{\infty}(0,T;H^{-1}(\Omega))} \leq C(T),\label{rho_f}
\eq 
\bq\|\phi^{\varepsilon}\|_{L^{\infty}(0,T;H^1(\Omega))} \leq C(T) ,\label{phi_L2H1}
\eq
\bq
{ 
\|\mu^{\varepsilon}\|_{L^{2}(0,T;H^1(\Omega))} \leq C(T),\label{mu_L2H1}}
\eq
\bq
\|\phi^{\varepsilon}\|_{B_{\infty,\infty}^{\frac{1}{4}}(0,T;L^2(\Omega))} \leq C(T),\label{phi_f}
\eq
\bq
\|\bfu^{\varepsilon}\|_{B_{2,\infty}^{\frac{1}{4}}(0,T;\bfH)} \leq C(T).\label{u_f}
\eq
\end{proposition}
\begin{proof}
First of all, we have the density bounds

\bq
\rho_* \leq \rho_{\varepsilon}(x, t) \leq \rho^* \quad \text { a.e. in } \Omega \times(0, T) .\label{density_bounds}
\eq

Following the same lines as in the proof of Proposition \ref{proof_prop_2}, applied to the regularized problem \ref{rho_epsilon}--\ref{muepsilon} we obtain

\bq
\begin{aligned}
\int_{\Omega}& \frac{1}{2} \rho^{\varepsilon}(t)\left|\bfu^{\varepsilon}(t)\right|^2 +\frac{1}{2}\Ga^2\left(\nabla \phi^{\varepsilon}(t)\right)+\rho^{\varepsilon}(t) F_{\varepsilon}\left(\phi^{\varepsilon}(t)\right) \md x\\
&+\int_0^t \int_{\Omega} \nu\left(\phi^{\varepsilon}(\tau)\right)\left|\mathbb{D} \bfu^{\varepsilon}(\tau)\right|^2 \md x \md \tau +\int_0^t \int_{\Omega}D(\phi^{\varepsilon}(\tau))\left|\nabla \mu^{\varepsilon}(\tau)\right|^2 \md x \md \tau\\
&=\int_{\Omega} \frac{1}{2} \rho_{0 }\left|\bfu_{0 }\right|^2+\frac{1}{2}\Ga^2\left(\nabla \phi_{0 }\right)+\rho_{0 } F_{\varepsilon}\left(\phi_{0 }\right) \mathrm{d} x:=E_0^{\varepsilon}.\label{energy_epsilon}
\end{aligned}
\eq
Since $\|\phi_0\|_{L^{\infty}} \leq 1$, and according to \eqref{F_propriety}, we have
\bq
E^{\varepsilon}_0 \leq \int_{\Omega} \frac{1}{2} \rho_{0} |u_{0}|^2 + \frac{1}{2}\Ga^2(\nabla\phi_{0}) + \rho_{0} F(\phi_{0}) \md x := E_0,\label{E_epsi<E_0}
\eq
{ 
Consequently, there exists a constant $C = C(\rho_*, \nu_*,D_*,r)$, but not depend on $\varepsilon$, such that
\bal
&\|\bfu^{\varepsilon}\|_{L^{\infty}(0,T;\bfH)}\leq C\sqrt{E_0},\quad \|\bfu^{\varepsilon}\|_{L^2(0,T;\bfV)}\leq C\sqrt{E_0},\label{u_L2_H_V}\quad\\
&\|\nabla\phi^{\varepsilon}\|_{L^{\infty}(0,T;L^2(\Omega))}\leq C\sqrt{E_0},
\quad\|\nabla\mu^{\varepsilon}\|_{L^2(0,T;L^2(\Omega))}\leq C\sqrt{E_0}.\label{mu_L2_2}
\eal
}
{  Thus, the inequalities \eqref{rho_f}, \eqref{phi_L2H1} and \eqref{phi_f} follows easily. For the equation \eqref{mu_L2H1}, from \eqref{mu_L2_2}, we have
\[
\|\na \mu^{\varepsilon}\|_{L^2(0,T;L^2(\Om))}\leq C\sqrt{E_0},
\]
By following the same lines of \cite[page 223]{giorgini2020weak} we obtain
\bq
\| \mu^{\varepsilon}\|_{L^2(0,T;H^1(\Om))}\leq C(T).
\eq
}

For the last inequality \eqref{u_f}, proceeding as in \eqref{time_translation}, we have
\bq
\begin{aligned}\label{time_translation_2}
 \int_{\Omega} \rho^{\varepsilon}(t+h)&\left|\bfu^{\varepsilon}(t+h)-\bfu^{\varepsilon}(t)\right|^2 \md x
 =\\ 
&\underbrace{\int_{\Omega}-\left(\rho^{\varepsilon}(t+h)-\rho^{\varepsilon}(t)\right) \bfu^{\varepsilon}(t) \cdot\left(\bfu^{\varepsilon}(t+h)-\bfu^{\varepsilon}(t)\right) \mathrm{d} x}_{J_1(t)} \\
& +\underbrace{\int_t^{t+h} \int_{\Omega}-\operatorname{div}\left(\rho^{\varepsilon}(\tau) \bfu^{\varepsilon}(\tau) \otimes \bfu^{\varepsilon}(\tau)\right) \cdot\left(\bfu^{\varepsilon}(t+h)-\bfu^{\varepsilon}(t)\right) \mathrm{d} x \md \tau}_{J_2(t)} \\
& +\underbrace{\int_t^{t+h} \int_{\Omega}-\nu\left(\phi^{\varepsilon}(\tau)\right) \mathbb{D} \bfu^{\varepsilon}(\tau): \nabla\left(\bfu^{\varepsilon}(t+h)-\bfu^{\varepsilon}(t)\right) \mathrm{d} x \md \tau}_{J_3(t)} \\
& +\underbrace{\int_t^{t+h} \int_{\Omega} \rho^{\varepsilon}(\tau) \mu^{\varepsilon}(\tau) \nabla \phi^{\varepsilon}(\tau) \cdot\left(\bfu^{\varepsilon}(t+h)-\bfu^{\varepsilon}(t)\right) \mathrm{d} x \md \tau}_{J_4(t)} \\
& +\underbrace{\int_t^{t+h} \int_{\Omega}-\rho^{\varepsilon}(\tau) F_{\varepsilon}^{\prime}\left(\phi^{\varepsilon}(\tau)\right) \nabla \phi^{\varepsilon}(\tau) \cdot\left(\bfu^{\varepsilon}(t+h)-\bfu^{\varepsilon}(t)\right) \mathrm{d} x \md \tau}_{J_5(t)}. 
\end{aligned}
\eq
Same here, $\int_{0}^{T-h} |J_k(t)| \md t \leq C h^{\frac{1}{2}}$ for $k=1,2,3\text{ and } 4$, with $C$ is a constant independent of $\varepsilon$. Thus, it is sufficient to control $J_5(t)$ by bounding $\left\|F_{\varepsilon}^{\prime}\left(\phi^{\varepsilon}\right)\right\|_{L^2(0,T;{L^6(\Omega)})}$ independently of $\varepsilon$. To this end, we show that
$$\left\|G_{\varepsilon}^{\prime}\left(\phi^{\varepsilon}\right)\right\|_{L^2(0,T;{L^6(\Omega)})} \leq C.$$
where $C$ is a constant independent of $\varepsilon$.
The argument is inspired from \cite{Abel}, we define for all $k>1$
$$
\phi_k^{\epsilon}=h_k \circ \phi^{\varepsilon}, \quad h_k(s)= \begin{cases}1-\frac{1}{k} & s>1-\frac{1}{k}, \\ s & -1+\frac{1}{k} \leq s \leq 1-\frac{1}{k}, \\ -1+\frac{1}{k} & s<-1+\frac{1}{k} .\end{cases}
$$

The chain rule holds $\nabla \phi_k^{\varepsilon}=\nabla \phi^{\varepsilon} \chi_{\left[-1+\frac{1}{k}, 1-\frac{1}{k}\right]}(\phi^{\varepsilon})$.

Now, multiplying \eqref{muepsilon} by $\left|G_{\varepsilon}^{\prime}\left(\phi_k^{\varepsilon}\right)\right|^{4} G_{\varepsilon}^{\prime}\left(\phi_k^{\varepsilon}\right)$ and integrating over $\Omega$, we find
$$
\begin{aligned}
\int_{\Omega}\left|G_{\varepsilon}^{\prime}\left(\phi_k^{\varepsilon}\right)\right|^{4} G_{\varepsilon}^{\prime \prime}\left(\phi_k^{\varepsilon}\right) \Ga\bfxi(\nabla \phi^{\varepsilon}) \cdot \nabla \phi_k^{\varepsilon} \md x &+\int_{\Omega} \rho^{\varepsilon}\left|G_{\varepsilon}^{\prime}\left(\phi_k^{\varepsilon}\right)\right|^{4} G_{\varepsilon}^{\prime}\left(\phi_k^{\varepsilon}\right) G_{\varepsilon}^{\prime}(\phi^{\varepsilon}) d x \\
&=\int_{\Omega}\left(\rho^{\varepsilon} \mu^{\varepsilon}+\lambda_1 \rho^{\varepsilon}\phi^{\varepsilon}\right)\left|G_{\varepsilon}^{\prime}\left(\phi_k^{\varepsilon}\right)\right|^{4} G_{\varepsilon}^{\prime}\left(\phi_k^{\varepsilon}\right) d x .
\end{aligned}
$$

According to the assumption $\textbf{H}_3$ and the fact that $G_{\varepsilon}$ is strictly convex, the first term on the left-hand side is non-negative. We also have that $G_{\varepsilon}^{\prime}\left(\phi_k^{\varepsilon}\right)^2 \leq G_{\varepsilon}^{\prime}(\phi^{\varepsilon}) G_{\varepsilon}^{\prime}\left(\phi_k^{\varepsilon}\right)$ almost everywhere. Thus, by Young's inequality, we obtain
$$
\left\|G_{\varepsilon}^{\prime}\left(\phi_k^{\varepsilon}\right)\right\|_{L^6(\Omega)}^6 \leq C\left\|\rho^{\varepsilon} \mu^{\varepsilon}+ \lambda_1\rho^{\varepsilon}\phi^{\varepsilon}\right\|_{L^6(\Omega)}^6,
$$
From  \eqref{phi_L2H1}, \eqref{mu_L2H1} and \eqref{density_bounds} we get $\rho^{\varepsilon} \mu^{\varepsilon}+ \lambda_1\rho^{\varepsilon}\phi^{\varepsilon}\in {L^2(0,T;{L^6(\Omega)})}$. Then, there exist a constant $C:=C(\rho_*,\rho^*,E_0,\lambda_1)$ independent of $\varepsilon$ and $k$ such that 
$$\left\|G_{\varepsilon}^{\prime}\left(\phi_k^{\varepsilon}\right)\right\|_{L^2(0,T;{L^6(\Omega)})} \leq C(T)$$ By applying Fatou's lemma, we have 
$$\left\|G_{\varepsilon}^{\prime}\left(\phi^{\varepsilon}\right)\right\|_{L^2(0,T;{L^6(\Omega)})} \leq C(T)$$
Since $F^{\prime} (\phi^{\varepsilon}) = -\lambda_1 \phi^{\varepsilon} + G^{\prime}_{\varepsilon}(\phi^{\varepsilon})$, we deduce that $\left\|F_{\varepsilon}^{\prime}\left(\phi^{\varepsilon}\right)\right\|_{L^2(0,T;{L^6(\Omega)})}$ bounded independently of $\varepsilon$, i.e
\bq
\left\|F_{\varepsilon}^{\prime}\left(\phi^{\varepsilon}\right)\right\|_{L^2(0,T;{L^6(\Omega)})} \leq C(T).\label{F_L2_L6}
\eq
{  

Recalling that \eqref{muepsilon} holds almost everywhere, we multiply this equation
by $\Delta \phi^\varepsilon $ and integrate over $\Omega$.
After integrating by parts and using the boundary condition for $\phi_\varepsilon$,
we find
\[
\int_\Om\operatorname{div}\left(\Ga\bfxi(\nabla\phi^{\varepsilon})\right)\De\phi^{\varepsilon} \md x= -\int_\Om\rho^{\varepsilon}\mu^{\varepsilon}\De\phi^{\varepsilon}\md x  + \int_\Om\rho^{\varepsilon} F_{\varepsilon}^{\prime}(\phi^{\varepsilon})\De\phi^{\varepsilon}\md x,
\]
thus,
\[
a_\Ga \int_{\Omega} \lvert D^2 \phi^{\varepsilon} \rvert^2
\le  \rho^* \|\mu^{\varepsilon}\|_{L^2(\Om)}\|\De\phi^{\varepsilon}\|_{L^2(\Om)}+\rho^*\left(\int_\Om |F_{\varepsilon}^{\prime}(\phi^{\varepsilon})|^2\,\md x\right)^{1/2}\|\De \phi^{\varepsilon}\|_{L^2(\Om)}
\]
According to \eqref{F_L2_L6} in particularly we have $\left\|F_{\varepsilon}^{\prime}\left(\phi^{\varepsilon}\right)\right\|_{L^2(0,T;{L^2(\Omega)})} \leq C(T)$.
We deduce that 
\bq
\|\phi^{\varepsilon}\|_{L^2(0,T;H^2(\Om))} \le C(T).\label{phi_L2_H2_epsi}
\eq

Again, by \cite[Lemma 3.1]{giorgini2020weak}, together with \eqref{phi_L2_H2_epsi}-- \eqref{phi_f} and \eqref{u_f}--\eqref{u_L2_H_V} we have 

\begin{align}
\bfu^{\varepsilon} &\to \bfu
&& \text{strongly in } L^2\bigl(0,T; L^q(\Omega)\bigr),
\quad \forall\, q \in [2,6),  \\[0.5em]
\phi^{\varepsilon} &\to \phi
&& \text{strongly in } L^2\bigl(0,T; W^{1,q}(\Omega)\bigr),
\quad \forall\, q \in [2,6), \\[0.5em]
\phi^{\varepsilon} &\to \phi 
&& \text{strongly in } C\bigl([0,T]; L^q(\Omega)\bigr),
\quad \forall\, q \in [2,6).
\end{align}

}
\end{proof}
{ 

{ 
By recalling the proposition \ref{prop12}, { same as \eqref{pass_limit}, we have  
\bq
\Gamma \bfxi(\nabla \phi^{\varepsilon}) \otimes \nabla \phi^{\varepsilon} \rightarrow \Gamma \bfxi(\nabla \phi^{}) \otimes \nabla \phi^{}  \text{ strongly in }  L^1(0,T; L^2(\Omega)).\eq
}
We thus conclude that there exist functions $(\rho,\bfu,\phi,\mu)$ satisfying}
}
\bal
& \rho \in \mathcal{C}([0,T];L^2(\Omega))\cap L^{\infty}(\Omega \times(0, T)) \cap W^{1, \infty}\left(0, T ; H^{-1}(\Omega)\right),\\
& \bfu \in L^2\left(0, T ; \bfV\right) \cap B_{2, \infty}^{\frac{1}{4}}\left(0, T ; \bfH\right), \\
& \phi \in  { L^{\infty}\left(0, T ; H^1(\Omega)\right)} \cap B_{\infty, \infty}^{\frac{1}{4}}\left(0, T ; L^2(\Omega)\right), \\
& \mu \in {
 L^{2}\left(0, T ; H^1(\Omega)\right)}.
\eal
and we thus complete the proof of Theorem \ref{th1}.

\section*{Declarations}


\textbf{Conflict of interest}: The authors have not disclosed any competing interests.






\begin{appendices}

\section{Bihari's Inequality}\label{A:Bihari's}

\begin{lemma}
Let $f: [0, +\infty) \to [0, +\infty)$ be a non-decreasing continuous function such that $f > 0$ on $]0, +\infty)$ and $\int_1^{+\infty} \frac{1}{f(x)} \, dx < +\infty$. We denote by $F$ the anti-derivative of $-1/f$ that vanishes at $+\infty$. Let $y$ be a continuous nonnegative function on $[0, +\infty)$, and let $g$ be a nonnegative function in $L_{loc}^1([0, +\infty))$. We assume that there exists a $y_0 > 0$ such that for all $t \geq 0$, we have the inequality

\[
y(t) \leq y_0 + \int_0^t g(s) \, ds + \int_0^t f(y(s)) \, ds.
\]

Then, there exists a unique $T^*$ that satisfies the equation
\[
T^* = F\left(y_0 + \int_0^{T^*} g(s) \, ds\right),
\]
and for any $T < T^*$, we have
\[
\sup_{t \leq T} y(t) \leq F^{-1}\left(F\left(y_0 + \int_0^T g(s) \, ds\right) - T\right).
\]
\end{lemma}
\cite[see page 90]{bihari1956generalization}.




\end{appendices}


\section*{Acknowledgments}

This work was funded by UM6P. The authors are deeply grateful to Professor Andrea Giorgini
(Department of Mathematics, Polytechnic University of Milan) for his insightful suggestions and invaluable guidance, which greatly contributed to this work. A. Zaidni also sincerely thanks Professor Alexis Vasseur (University of Texas at Austin) for his valuable discussions and suggestions.


\begin{thebibliography}{0}

\bibitem{Abel}
H. Abels, On a diffuse interface model for two-phase flows of viscous, incompressible fluids with matched densities, {\it Arch. Ration. Mech. Anal.} {\bf 194} (2009) 463--506.

\bibitem{abels2013}
H. Abels, D. Depner and H. Garcke, On an incompressible Navier--Stokes/Cahn--Hilliard system with degenerate mobility, {\it Ann. Inst. H. Poincar\'e C} {\bf 30} (2013) 1175--1190.

\bibitem{abels2024}
H. Abels, H. Garcke and A. Giorgini, Global regularity and asymptotic stabilization for the incompressible Navier--Stokes--Cahn--Hilliard model with unmatched densities, {\it Math. Ann.} {\bf 389} (2024) 1267--1321.

\bibitem{anderson98}
D. M. Anderson, G. B. McFadden and G. B. Wheeler, Diffuse-interface methods in fluid mechanics, {\it Annu. Rev. Fluid Mech.} {\bf 30} (1998) 139--165.

\bibitem{anderson2000phase}
D. M. Anderson, G. B. McFadden and G. B. Wheeler, A phase-field model of solidification with convection, {\it Physica D} {\bf 135} (2000) 175--194.

\bibitem{bihari1956generalization}
I. Bihari, A generalization of a lemma of Bellman and its application to uniqueness problems of differential equations, {\it Acta Math. Hung.} {\bf 7} (1956) 81--94.

\bibitem{cahn1958free}
J. W. Cahn and J. E. Hilliard, Free energy of a nonuniform system. I. Interfacial free energy, {\it J. Chem. Phys.} {\bf 28} (1958) 258--267.

\bibitem{Chen}
Y. Chen, Q. He, M. Mei and X. Shi, Asymptotic stability of solutions for 1-D compressible Navier-Stokes-Cahn-Hilliard system, {\it J. Math. Anal. Appl.} {\bf 467} (2018) 185--206.

\bibitem{Cherfils}
L. Cherfils, E. Feireisl, A. Miranville and M. Paicu, The compressible Navier-Stokes-Cahn-Hilliard equations with dynamic boundary conditions, {\it Math. Models Methods Appl. Sci.} {\bf 29} (2019) 2557--2584.

\bibitem{cherfils2019compressible}
L. Cherfils, E. Feireisl, M. Mich\'alek, A. Miranville, M. Petcu and D. Pra\v{z}\'ak, The compressible Navier--Stokes--Cahn--Hilliard equations with dynamic boundary conditions, {\it Math. Models Methods Appl. Sci.} {\bf 29} (2019) 2557--2584.

\bibitem{dilmi2025}
M. Dilmi, M. Kirane, M. Dilmi and H. Benseridi, Asymptotic behavior of Navier-Stokes-Voigt equations in a thin domain with damping term and Tresca friction law, {\it J. Hyperbolic Differ. Equ.} {\bf 22} (2025) 591--612.

\bibitem{dingwell}
S. Ding and Y. Li, Well-posedness for 1D Compressible Navier-Stokes/Cahn-Hilliard system.

\bibitem{dlotko2022navier}
T. Dlotko, Navier-Stokes-Cahn-Hilliard system of equations, {\it J. Math. Phys.} {\bf 63} (2022) 111511.

\bibitem{elbar2024}
C. Elbar and A. Poulain, Analysis and numerical simulation of a generalized compressible Cahn--Hilliard--Navier--Stokes model with friction effects, {\it ESAIM: M2AN} {\bf 58} (2024) 1989--2034.

\bibitem{Garcke}
H. Garcke, P. Knopf and A. Signori, The anisotropic Cahn--Hilliard equation with degenerate mobility: existence of weak solutions, {\it Anal. Appl.} (2025) 1--23.

\bibitem{garcke2023}
H. Garcke, P. Knopf and J. Wittmann, The anisotropic Cahn--Hilliard equation: regularity theory and strict separation properties, arXiv:2305.18255 (2023).

\bibitem{giorgini2023existence}
A. Giorgini, A. Ndongmo Ngana, T. T. Medjo and R. Temam, Existence and regularity of strong solutions to a nonhomogeneous Kelvin-Voigt-Cahn-Hilliard system, {\it J. Differential Equations} {\bf 372} (2023) 612--656.

\bibitem{giorgini2020weak}
A. Giorgini and R. Temam, Weak and strong solutions to the nonhomogeneous incompressible Navier-Stokes-Cahn-Hilliard system, {\it J. Math. Pures Appl.} {\bf 144} (2020) 194--249.


\bibitem{giorgini2021navier}
A. Giorgini, R. Temam and X.-T. Vu, The Navier-Stokes-Cahn-Hilliard equations for mildly compressible binary fluid mixtures, {\it Discrete Contin. Dyn. Syst. Ser. B} {\bf 26} (2021).

\bibitem{Jiaojiao}
J. Pan, C. Xu and H. Liu, Uniform regularity of the weak solution to higher-order Navier-Stokes-Cahn-Hilliard systems, {\it J. Math. Anal. Appl.} {\bf 486} (2020).

\bibitem{kotschote2015strong}
M. Kotschote and R. Zacher, Strong solutions in the dynamical theory of compressible fluid mixtures, {\it Math. Models Methods Appl. Sci.} {\bf 25} (2015) 1217--1256.

\bibitem{li2024strong}
Y. Li, Y. Wang and H. Cai, Strong solutions of an incompressible phase-field model with variable density, {\it AMASES} (2024).

\bibitem{munteanu2024}
I. Munteanu, Well-posedness for the Cahn-Hilliard-Navier-Stokes equations perturbed by gradient-type noise, {\it Appl. Math. Optim.} {\bf 89} (2024).

\bibitem{Pierluigi}
P. Colli, S. Frigeri and M. Grasselli, Global existence of weak solutions to a nonlocal Cahn--Hilliard--Navier--Stokes system, {\it J. Math. Anal. Appl.} {\bf 386} (2012) 428--444.

\bibitem{rui2024global}
N. Rui, F. Li and Z. Guo, Global well-posedness of the nonhomogeneous incompressible Navier-Stokes-Cahn-Hilliard system with Landau potential, arXiv:2409.11775 (2024).

\bibitem{taylor-cahn98}
J. E. Taylor and J. W. Cahn, Diffuse interfaces with sharp corners and facets: Phase field models with strongly anisotropic surfaces, {\it Physica D} {\bf 112} (1998) 194--249.



\bibitem{temam2012infinite}
R. Temam, {\it Infinite-Dimensional Dynamical Systems in Mechanics and Physics} Springer-Verlag, New York, (1997).

\bibitem{tzavaras2024}
A. E. Tzavaras, Oscillations in compressible Navier-Stokes and homogenization in phase transition problems, {\it J. Hyperbolic Differ. Equ.} {\bf 21} (2024) 827--843.

\bibitem{zaidni2024}
A. Zaidni, P. J. Morrison and S. Benjelloun, Thermodynamically consistent Cahn-Hilliard-Navier-Stokes equations using the metriplectic dynamics formalism, {\it Physica D} {\bf 468} (2024) 134--303.

\bibitem{Zhao}
X. Zhao, Strong solutions to the density-dependent incompressible Cahn-Hilliard-Navier-Stokes system, {\it J. Hyperbolic Differ. Equ.} {\bf 16} (2019) 701--742.

\end{thebibliography}
\end{document}